\documentclass[a4paper, 11pt]{amsart}

\usepackage{amsfonts,amsmath,amssymb,amsxtra,amsthm,tikz}
\usepackage{mathrsfs}
\usepackage{times}
\usepackage{anysize}
\usepackage{graphicx}
\marginsize{1in}{1in}{1in}{1in}

\input xy
\xyoption{all}
\usepackage[all,knot]{xy}

\newcommand{\Tot}{\operatorname{{\bf Tot}}}
\newcommand{\cw}{\operatorname{cw}}
\newcommand{\Ext}{\operatorname{Ext}}
\newcommand{\rep}{\operatorname{rep}}

\newcommand{\Hom}{\operatorname{Hom}}
\newcommand{\Iso}{\operatorname{Iso}}
\newcommand{\End}{\operatorname{End}}
\renewcommand{\mod}{\operatorname{mod}\nolimits}
\newcommand{\add}{\operatorname{add}}
\newcommand{\proj}{\operatorname{proj}}
\newcommand{\forget}{\operatorname{for}}
\newcommand{\qis}{\operatorname{qis}}
\newcommand{\Fac}{\operatorname{Fac}}
\newcommand{\mt}{\operatorname{mt}}

\newtheorem{theorem}{Theorem}[section]
\newtheorem{corollary}[theorem]{Corollary}
\newtheorem{lemma}[theorem]{Lemma}
\newtheorem{proposition}[theorem]{Proposition}

\newtheorem{warning}[theorem]{Warning}

\theoremstyle{definition}
\newtheorem{definition}[theorem]{Definition}

\newtheorem{remark}[theorem]{Remark}

\title{Semi-derived Hall algebras and tilting invariance of Bridgeland--Hall algebras} 
\author{Mikhail Gorsky} 

\address{Universit\'e Paris Diderot -- Paris 7, UFR de
Math\'ematiques, Case 7012, Institut de Math\'ematiques de Jussieu -- Paris Rive Gauche, UMR 7586 du CNRS,
B\^at. Sophie Germain, 75205 Paris Cedex 13, France}
\address{Steklov Mathematical Institute,  
8 Gubkina Street, Moscow, Russia 119991.}
\email{gorskym@math.jussieu.fr} 

\begin{document}

\begin{abstract}
Inspired by recent work of Bridgeland, from the category $\mathcal{C}^b(\mathcal{E})$ of bounded complexes over an exact category $\mathcal{E}$ satisfying certain finiteness conditions, we construct an associative unital {\it ``semi-derived Hall algebra''} $\mathcal{SDH}(\mathcal{E}).$ This algebra is an object sitting, in some sense, between the usual Hall algebra $\mathcal{H}(\mathcal{C}^b(\mathcal{E}))$ and the Hall algebra of the bounded derived category $\mathcal{D}^b(\mathcal{E}),$ introduced by To\"{e}n and further generalized by Xiao and Xu. It has the structure of a free module over a suitably defined quantum torus of acyclic complexes, with a basis given by the isomorphism classes of objects in the bounded derived category $\mathcal{D}^b(\mathcal{E}).$ We prove the invariance of $\mathcal{SDH}(\mathcal{E})$ under derived equivalences induced by exact functors between exact categories. 

For $\mathcal{E}$ having enough projectives and such that each object has a finite projective resolution, we describe a similar construction for the category of $\mathbb{Z}/2-$graded complexes, with similar properties of associativity, freeness over the quantum torus and derived invariance. In particular, we obtain that this $\mathbb{Z}/2-$graded semi-derived Hall algebra is isomorphic to the two-periodic Hall algebra recently introduced by Bridgeland. We deduce that Bridgeland's Hall algebra is preserved under tilting.

When $\mathcal{E}$ is hereditary and has enough projectives, we show that the multiplication in $\mathcal{SDH}(\mathcal{E})$ is given by the same formula as the Ringel-Hall multiplication, and $\mathcal{SDH}(\mathcal{E})$ is isomorphic to a certain quotient of the classical Hall algebra $\mathcal{H}(\mathcal{C}^b(\mathcal{E}))$ localized at the classes of acyclic complexes. We also prove the same result in the $\mathbb{Z}/2-$graded case.
\end{abstract}


\maketitle

\tableofcontents

\section{Introduction}
Hall algebras of abelian categories are associative algebras encoding extensions in these categories. They were introduced by Ringel in \cite{R1}. He developed ideas of Hall \cite{Hal} who worked with the abelian category of commutative finite $p-$groups, and even older work of Steinitz \cite{St}. Ringel proved that the twisted Hall algebra $\mathcal{H}_{tw}(\rep_k (Q))$ of the category of representations of a simply-laced Dynkin quiver over the finite field $k = \mathbb{F}_q$ is isomorphic to the nilpotent part $U_t(\mathfrak{n}^{+})$ of the universal enveloping algebra of the corresponding quantum group $\mathfrak{g}.$ Here $t = +\sqrt{q}.$ Later Green generalized this result in \cite{G}, constructing an embedding 
$$U_t(\mathfrak{n}_+) \hookrightarrow \mathcal{H}_{tw}(\rep_k (Q)),$$
for $Q$ an arbitrary acyclic quiver, where $\mathfrak{n}_+$ is the positive nilpotent part of the corresponding derived Kac-Moody algebra. Similarly, there is an extended version $\mathcal{H}^{e}_{tw}(\rep_k (Q))$ of the Hall algebra, which admits an embedding 
$$U_t(\mathfrak{b}^{+}) \hookrightarrow \mathcal{H}^{e}_{tw}(\rep_k (Q)),$$
where $\mathfrak{b}^{+}$ is the positive Borel subalgebra of the corresponding derived Kac-Moody algebra.

Hall algebras were generalized to the framework of exact categories (in sense of Quillen \cite{Q}), see \cite{Hu}. A typical example of an exact category is the subcategory $\mathcal{P}$ of projective objects in an abelian category. For the basics of exact categories, we refer to \cite{Buh} \cite[App.]{Kel1} \cite{Kel2}.

Another generalization was given by To\"{e}n \cite{T}, who constructed an associative unital algebra for each dg-category satisfying some finiteness conditions. This is called a {\it derived Hall algebra}. In \cite{XX1}, Xiao and Xu showed that To\"{e}n's construction works for all triangulated categories satisfying certain finiteness conditions. In both cases, the most important example is the derived category of a suitable exact category. 
Another approach to constructing Hall algebras of derived categories was proposed before by Kapranov \cite{Kap}, who considered a $\mathbb{Z}/2-$graded version of these categories; see also Cramer's \cite{C}. 

One important motivation for Kapranov's and To\"{e}n's work was the hope to extend the above picture so as to obtain a Hall algebra description of the {\it whole} quantum group $U_t(\mathfrak{g}).$ This hope was finally realized in Bridgeland's paper \cite{B}. Starting from the twisted Hall algebra $\mathcal{H}_{tw}(\mathcal{C}_{\mathbb{Z}/2}(\mathcal{P}_{\mathcal{A}}))$ of the category of $\mathbb{Z}/2-$graded complexes with projective components of a hereditary abelian category $\mathcal{A},$ he defined the associative algebra $\mathcal{DH}_{red}(\mathcal{A})$ as a quotient of the localization of $\mathcal{H}_{tw}(\mathcal{C}_{\mathbb{Z}/2}(\mathcal{P}_{\mathcal{A}}))$ with respect to all contractible complexes.
He constructed an embedding from $U_t(\mathfrak{g})$ into $\mathcal{DH}_{red}(\mathcal{A}),$ where $\mathcal{A}$ is the category of quiver representations; it is an isomorphism exactly in the Dynkin case, as in Ringel-Green's theorem. 
Thus, in this case, the algebra $\mathcal{DH}_{red}(\mathcal{A})$ is a reduced Drinfeld double. In fact, more generally, Bridgeland stated and Yanagida proved \cite{Y} that $\mathcal{DH}_{red}(\mathcal{A})$ is always the reduced Drinfeld double of the extended twisted Hall algebra $\mathcal{H}_{tw}^{e}(\mathcal{A}).$
The existence of a {\it generic} Bridgeland-Hall algebra for the category of modules over a finite-dimensional representation-finite hereditary algebra was recently shown by Chen and Deng \cite{CD}.

This article is motivated by the problem of showing that Bridgeland's construction is invariant under derived equivalences. As a first step towards this goal, we prove that Bridgeland--Hall algebras are preserved under tilting: if two algebras (of finite dimension and finite global dimension) are related by a tilting triple \cite{Bon}\cite{HR}, the Bridgeland--Hall algebras of their categories of finitely generated modules are canonically isomorphic (Corollary~\ref{z2tiltequiv}). 
We also prove a similar result for categories of coherent sheaves admitting a tilting bundle, but for a $\mathbb{Z}-$graded version of the Bridgeland--Hall algebra (Theorem~\ref{tiltequiv}). Our main tool in the proof is a new type of Hall algebra, the {\em semi-derived Hall algebra}.

Here is a more detailed description of the contents of this paper. Starting from the category of bounded complexes $\mathcal{C}^b(\mathcal{E})$ over an exact category satisfying certain conditions, we construct an associative unital algebra. In Section 3, we describe these conditions and prove that they hold in two important cases: either when $\mathcal{E}$ is the category of coherent sheaves on a smooth projective variety, or when $\mathcal{E}$ has enough projectives and each object has a finite projective resolution.
We start the construction in Section 4 by defining a quantum torus of acyclic complexes. In Section 5, we consider a module over this torus, generated by the classes of all complexes, factorized by certain relations. It turns out to be a free module with a basis indexed by the isomorphism classes of objects in the bounded derived category $\mathcal{D}^b(\mathcal{E}).$ We define a multiplicative structure on this module in a way slightly similar to the usual Hall product.  We verify that this multiplication is associative in Section 6.1. 
Since the resulting algebra remembers the quantum torus of acyclics as well as the structure of the derived category, we call it the {\it semi-derived Hall algebra}  $\mathcal{SDH}(\mathcal{E}).$

In Section 6, we prove the invariance of $\mathcal{SDH}(\mathcal{E})$ under derived equivalences induced by exact functors between exact categories. If $\mathcal{E}$ has enough projectives and finite projective dimension, we prove in Section 7 that $\mathcal{SDH}(\mathcal{E})$ is isomorphic to the $\mathbb{Z}-$graded version of Bridgeland's algebra. We construct the isomorphisms in both ways. If $\mathcal{E}$ is hereditary and has enough projectives, we show that the multiplication in $\mathcal{SDH}(\mathcal{E})$ is given by the same formula as the usual Ringel-Hall multiplication. From here it immediately follows that $\mathcal{SDH}(\mathcal{E})$ is isomorphic to a certain quotient of the classical Hall algebra $\mathcal{H}(\mathcal{C}^b(\mathcal{E}))$ localized at the classes of all acyclic complexes, cf. Corollary \ref{herquotloc}. 
In Section 8, we deduce that the semi-derived Hall algebra is invariant under tilting equivalences.

In Section 9, we describe a similar construction in the $\mathbb{Z}/2-$graded case. For the moment, we can do it only in the case when $\mathcal{E}$ has enough projectives and each object has a finite projective resolution. We have to start not from the entire category $\mathcal{C}_{\mathbb{Z}/2}(\mathcal{E}),$ but from a certain exact subcategory $\widetilde{\mathcal{E}}.$ We construct an associative unital algebra $\mathcal{SDH}_{\mathbb{Z}/2}(\mathcal{E})$ with a free module structure over the quantum torus of acyclic $\mathbb{Z}/2-$graded complexes, with a basis indexed by the quasi-isomorphism classes of objects in this exact subcategory. We prove its invariance under the same class of equivalences as in the $\mathbb{Z}-$graded case. In particular, this algebra with a twisted multiplication is isomorphic to Bridgeland's algebra and, therefore, for a hereditary abelian initial category $\mathcal{A},$ it realizes the Drinfeld double of  $\mathcal{H}^{e}_{tw}(\mathcal{A}).$ 
This fact combined with our result on derived invariance of the $\mathbb{Z}/2-$graded semi-derived Hall algebra yields a new proof of Cramer's theorem \cite[Theorem~1]{C} stating that derived equivalences between hereditary abelian categories give rise to isomorphisms between the Drinfeld doubles of their twisted and extended Hall algebras. 
For a hereditary initial category $\mathcal{E},$ we prove the same result as in the $\mathbb{Z}-$graded case: the multiplication in $\mathcal{SDH}_{\mathbb{Z}/2}(\mathcal{E})$ is given by the same formula as the usual Ringel-Hall multiplication, and $\mathcal{SDH}_{\mathbb{Z}/2}(\mathcal{E})$ is isomorphic to a certain quotient of the classical Hall algebra $\mathcal{H}(\widetilde{\mathcal{E}})$ localized at the classes of all acyclic $\mathbb{Z}/2-$graded complexes.
In Bridgeland's work, the images of the natural generators of the quantum group have a slightly mysterious form of products of the inverses of the classes of certain acyclic complexes by the classes of the minimal projective resolutions of the corresponding simples in $\rep_k (Q).$ Under the isomorphism with $\mathcal{SDH}_{\mathbb{Z}/2}(\mathcal{E}),$ they map to the shifted simples in $\rep_k (Q).$ This clarifies Bridgeland's construction as well as a result due to Sevenhant-Van den Bergh \cite{SV} and Xu-Yang \cite{XY} stating that Bernstein-Gelfand-Ponomarev reflection functors induce the braid group action on the quantum group first discovered by Lusztig \cite{L}. 

We expect that there are strong relations between our semi-derived Hall algebras and the algebras of Kapranov and To\"{e}n-Xiao-Xu.



This is an extended and edited version of my M.Sc.~Thesis at Universit\'{e} Paris Diderot. I am very grateful to my scientific adviser Prof. Bernhard Keller for posing me the problem and for his giant support. The work was supported by Fondation Sciences math\'{e}matiques de Paris and by R\'{e}seau de Recherche Doctoral en Math\'{e}matiques de l'\^{I}le de France.

\section{Preliminaries}


\subsection{Hall algebras}

Let $\mathcal{E}$ be an essentially small exact category, linear over a finite field $k.$ Assume that $\mathcal{E}$ has finite morphism and (first) extension spaces:
$$|\mbox{Hom}(A, B)| < \infty, \quad  |\Ext^1 (A, B)| < \infty, \quad \forall A, B \in \mathcal{E}.$$
Given objects $A, B, C \in \mathcal{A},$ define $\Ext^1 (A, C)_B \subset \Ext^1 (A, C)$ as the subset parameterizing extensions whose middle term is isomorphic to B. We define the Hall algebra $\mathcal{H(E)}$ to be the $\mathbb{Q}-$vector space whose basis is formed by the isomorphism classes $[A]$ of objects $A$ of $\mathcal{E},$
with the multiplication defined by
$$[A] \diamond [C] = \sum\limits_{B \in \mbox{\footnotesize{Iso}}(\mathcal{E})} \frac{|\Ext^1 (A, C)_B|}{|\mbox{Hom} (A, C)|} [B].$$

The following result was proved by Ringel \cite{R1} for $\mathcal{E}$ abelian, and later by Hubery \cite{Hu} for $\mathcal{E}$ exact. The definition of $\mathcal{H(E)}$ is also due to Ringel.

\begin{theorem}
The algebra $\mathcal{H(E)}$ is associative and unital. The unit is given by $[0]$, where $0$ is the zero object of $\mathcal{E}$.
\end{theorem}

\begin{remark}
The choice of the structure constants
$\frac{|\Ext^1_{\mathcal{E}}(A,B)_{C}|}{|\Hom_{\mathcal{E}}(A,B)|} $ is the one that was used by Bridgeland \cite{B}.
This choice is equivalent to that of the usual structure constants
$|\{ B' \subset C | B' \cong B, C/B' \cong A  \}|,$
called the {\it Hall numbers} and appearing in \cite{R1},\cite{Sch} and \cite{Hu}.
See \cite[\S2.3]{B} for the details.
\end{remark}

Assume that $\mathcal{E}$ is locally homologically finite and that all higher extension spaces are finite:
$$\forall A, B \in \mathcal{E} \quad \exists p_0: \quad \Ext^p (A, B) = 0, \quad \forall p > p_0;$$
$$|\Ext^p (A, B)| < \infty, \quad \forall p \geq 0, \quad \forall A, B \in \mathcal{E}.$$
For objects $A, B \in \mathcal{E}$, we define the Euler form 
$$\left\langle A, B \right\rangle := \prod_{i \in \mathbb{Z}} |\Ext^i_{\mathcal{E}}(A,B)|^{{(-1)}^i}.$$
It is well known and easy to check that this form descends to a bilinear form
on the Grothendieck group $K_0(\mathcal{E})$ of $\mathcal{E}$,
denoted by the same symbol:
$$\left\langle \cdot, \cdot \right\rangle K_0(\mathcal{E}) \times K_0(\mathcal{E}) \to \mathbb{Q}^{\times}.$$

The {\it twisted Hall algebra} $\mathcal{H}_{tw}(\mathcal{E})$ 
is the same vector space as $\mathcal{H}(\mathcal{E})$  
with the twisted multiplication
\begin{equation}
[A] * [B] := \sqrt{\left\langle A, B \right\rangle} \cdot [A] \diamond [B], \quad \forall  A, B \in \Iso(\mathcal{E}).
\end{equation}

The {\it twisted extended Hall algebra} $\mathcal{H}^{e}_{tw}(\mathcal{E})$ 
is defined as the extension of $\mathcal{H}_{tw}(\mathcal{E})$
obtained by adjoining symbols $K_\alpha$ for all classes $\alpha\in K_0(\mathcal{E}),$ 
and imposing relations
$$
K_\alpha * K_\beta = K_{\alpha+\beta}, \quad 
K_\alpha * [B]     = \sqrt{(\alpha, B)} \cdot [B] * K_\alpha,
$$
for $\alpha,\beta \in K_0(\mathcal{E})$ 
and $B \in \Iso(\mathcal{E})$.
Of course, we can construct $\mathcal{H}^{e}_{tw}(\mathcal{E})$ from $\mathcal{H} (\mathcal{E})$ in a different order: first adjoin symbols $K_\alpha$ and impose relations 
$$
K_\alpha \diamond K_\beta = \frac{1}{\sqrt{\left\langle \alpha, \beta \right\rangle}} K_{\alpha+\beta}, \quad 
K_\alpha \diamond [B] =  [B] \diamond K_\alpha,
$$
defining the {\it extended Hall algebra} $\mathcal{H}^{e}(\mathcal{E}),$ and then twist the multiplication by the Euler form.
Note that $\mathcal{H}^{e}(\mathcal{E})$ has by definition a module structure over the twisted group algebra of $K_0(\mathcal{E}),$ where the multiplication in the last one is twisted by the inverse of the Euler form.

There are famous results by Green and Xiao stating that if $\mathcal{A}$ is abelian and hereditary, then $\mathcal{H}^{e}_{tw}(\mathcal{A})$ admits a structure of a self-dual Hopf algebra (see {\cite[Sections~1.4-1.7]{Sch}} and references therein).

\subsection{$\mathbb{Z}/2-$graded complexes}

Let $\mathcal{C}_{\mathbb{Z}/2}(\mathcal{E})$ be the exact category 
of $\mathbb{Z}/2$-graded complexes over $\mathcal{E}$.  
Namely, an object $M$ of this category is a diagram with objects and morphisms in $\mathcal{E}$:
$$\xymatrix{M^0 \ar@<0.5ex>[r]^{d^0} & M^1 \ar@<0.5ex>[l]^{d^1}}, \quad d^1 \circ d^0 = d^0 \circ d^1 = 0.$$
All indices of components of $\mathbb{Z}/2-$graded objects will be understood modulo $2$.
A {\it morphism} $s: M \to N$ is a diagram
$$\xymatrix@R=0.6cm{
  *+++{M^0} \ar@{->}[d]_{s^0} \ar@<.5ex>[r]^{d^0} 
&     *+++{M^1} \ar@<.5ex>[l]^{d^1} \ar@{->}[d]^{s^1}
\\
  *+++{N^0}\ar@<.5ex>[r]^{d'^0}
&    *+++{N^1} \ar@<.5ex>[l]^{d'^1}}$$
with $s^{i+1} \circ d^{i} = d'^i\circ s^{i}$.
Two morphisms $s, t: M \to N$ are {\it homotopic} if there are morphisms $h^i: M^i\to N^{i+1}$ such that
$$t^i - s^i = d'^{i+1} \circ h^{i} + h^{i+1} \circ d^{i}.$$


Denote by $\mathcal{K}_{\mathbb{Z}/2}(\mathcal{E})$ the category 
obtained from $\mathcal{C}_{\mathbb{Z}/2}(\mathcal{E})$ 
by identifying homotopic morphisms.
Let us also denote by $\mathcal{D}_{\mathbb{Z}/2}(\mathcal{E})$ the $\mathbb{Z}/2-$graded derived category, i.e. the localization of $\mathcal{K}_{\mathbb{Z}/2}(\mathcal{E})$ with respect to all quasi-isomorphisms. 

The shift functor $\Sigma$ of complexes induces involutions
$$\mathcal{C}_{\mathbb{Z}/2}(\mathcal{E}) \stackrel{*}{\longleftrightarrow} \mathcal{C}_{\mathbb{Z}/2}(\mathcal{E}),\quad \mathcal{K}_{\mathbb{Z}/2}(\mathcal{E}) \stackrel{*}{\longleftrightarrow} \mathcal{K}_{\mathbb{Z}/2}(\mathcal{E}), \quad \mathcal{D}_{\mathbb{Z}/2}(\mathcal{E}) \stackrel{*}{\longleftrightarrow} \mathcal{D}_{\mathbb{Z}/2}(\mathcal{E}).$$
These involutions shift the grading and change the sign of the differential as follows:
$$\xymatrix{M^0 \ar@<0.5ex>[r]^{d^0} & M^1 \ar@<0.5ex>[l]^{d^1}}  \stackrel{*}{\longleftrightarrow} \xymatrix{M^1 \ar@<0.5ex>[r]^{-d^1} & M^0 \ar@<0.5ex>[l]^{-d^0}}.$$

We have an exact functor $$\pi: \mathcal{C}^b(\mathcal{E}) \to \mathcal{C}_{\mathbb{Z}/2}(\mathcal{E}),$$ sending a complex $(M^i)_{i \in \mathbb{Z}}$ to the $\mathbb{Z}/2-$graded complex 
\begin{equation} \label{impi}
\xymatrix{{\bigoplus\limits_{i \in \mathbb{Z}} M^{2i}} \ar@<0.5ex>[r] & {\bigoplus\limits_{i \in \mathbb{Z}} M^{2i+1}} \ar@<0.5ex>[l]}
\end{equation} 
with the naturally defined differentials. It is easy to check that 
$$\Hom_{\mathcal{C}_{\mathbb{Z}/2}(\mathcal{E})} (\pi(A), \pi(B)) = \bigoplus\limits_{i \in \mathbb{Z}} \Hom_{\mathcal{C}^b(\mathcal{E})} (A, \Sigma^{2i} B).$$
Note that we actually have a whole family of isomorphisms
\begin{equation} \label{piext}
\Ext^p_{\mathcal{C}_{\mathbb{Z}/2}(\mathcal{E})} (\pi(A), \pi(B)) = \bigoplus\limits_{i \in \mathbb{Z}} \Ext^p_{\mathcal{C}^b(\mathcal{E})} (A, \Sigma^{2i} B), 
\quad \forall p \geq 0.
\end{equation}

Similarly, we have exact functors
$$\mathcal{K}^b(\mathcal{E}) \to \mathcal{K}_{\mathbb{Z}/2}(\mathcal{E}), \quad \mathcal{D}^b(\mathcal{E}) \to \mathcal{D}_{\mathbb{Z}/2}(\mathcal{E}),$$
satisfying analogous isomorphisms. By abuse of notation, we will also denote them by $\pi.$ 



To each object $M \in\mathcal{E}$, we attach a pair of acyclic (in fact, even contractible) complexes
$$K_M := \xymatrix{M \ar@<0.5ex>[r]^{1} & M \ar@<0.5ex>[l]^{0}}, \qquad K_M^* := \xymatrix{M \ar@<0.5ex>[r]^{0} & M \ar@<0.5ex>[l]^{1}}.$$
Let $\mathcal{P}$ be the full subcategory of projective objects in $\mathcal{E}.$ The following fact was shown in \cite{B}.

\begin{lemma}[{\cite[Lemma~3.2]{B}}] \label{acprojdirsum}
Suppose that each object in $\mathcal{E}$ has a finite projective resolution. Then for any acyclic complex of projectives ${M} \in \mathcal{C}(\mathcal{P})$, there are objects $P,Q\in \mathcal{P}$, unique up to isomorphism, such that ${M} \cong {K_P} \oplus {K_Q}^*$. 
\end{lemma}

\begin{remark}
In \cite{B} this Lemma was proved only for $\mathcal{E}$ abelian and of finite global dimension; nonetheless, the same proof works in the generality stated above.
\end{remark}



\section{Finiteness conditions}
In this section, we will discuss the various conditions we impose on the exact category $\mathcal{E}.$ We shall always assume that
\begin{itemize}
\item[(C1)] $\mathcal{E}$ is essentially small, idempotent complete and linear over some ground field $k;$
\item[(C2)] For each pair of objects $A, B \in \mbox{Ob}(\mathcal{E})$ and for each $p > 0,$ we have
$$|\Ext^p (A, B)| < \infty; \quad \quad |\Hom (A, B)| < \infty;$$
\item[(C3)] For each pair of objects $A, B \in \mbox{Ob}(\mathcal{E}),$ there exists $N > 0$ such that for all $p > N,$ we have
$$\Ext^p (A, B) = 0.$$
\end{itemize}

The first part of the assumption (C1) and the finiteness of the morphism spaces are important to our naive approach to Hall algebras involving counting isomorphism classes. Assumption (C1) and $\Hom-$finiteness ensure that $\mathcal{E}$ is Krull-Schmidt. Moreover, it implies that all contractible complexes are acyclic, cf. \cite{Kel2}, \cite{Buh}. Finiteness of $\Ext^p (A, B)$ and the assumption (C3) are crucial for the multiplicative version of the Euler form that we use. These conditions concern the structure of our exact category $\mathcal{E}.$ Our construction is also based on a less natural assumption concerning the category $\mathcal{C}^b(\mathcal{E}),$ endowed with the component-wise exact structure: 

\begin{itemize}
\item[(C4)]
For each triple of bounded complexes $A, B, C \in \mathcal{C}^b(\mathcal{E}),$ there exists a pair of conflations

$$K_1 \rightarrowtail L_1 \stackrel{\qis}{\twoheadrightarrow} A, \quad K_2 \rightarrowtail L_2 \stackrel{\qis}{\twoheadrightarrow} B,$$

such that the following conditions hold:

\begin{itemize}
\item[(i)] 
\begin{equation} \label{extp}
\Ext^p_{\mathcal{C}^b(\mathcal{E})}(L_1, B) \cong \Ext^p_{\mathcal{D}^b(\mathcal{E})}(L_1, B),
\end{equation}

$$\quad \Ext^p_{\mathcal{C}^b(\mathcal{E})}(L_2, C) \cong \Ext^p_{\mathcal{D}^b(\mathcal{E})}(L_2, C), \quad \forall p > 0;$$

\item[(ii)]
$$\Ext^1_{\mathcal{C}^b(\mathcal{E})}(L_1, L_2) \overset\sim\to \Ext^1_{\mathcal{C}^b(\mathcal{E})}(L_1, B), \quad  \Ext^1(L_1, K_2) = 0;$$

\item[(iii)] for each conflation $$L_2 \rightarrowtail Y \twoheadrightarrow L_1$$ 
and all $p > 0,$ we have
$$\Ext^p_{\mathcal{C}^b(\mathcal{E})}(Y, C) \cong \Ext^p_{\mathcal{D}^b(\mathcal{E})}(Y, C).$$
\end{itemize}
\end{itemize}

We use the assumption (i) to construct the product in our algebra and to prove that it is well-defined. The whole condition (C4) is essential for proving the associativity of the product thus defined. It seems quite artificial, but it actucally holds in at least two natural types of examples. 

\begin{theorem} \label{condd}
The condition (C4) holds if $\mathcal{E}$ is of one of the following two types:
\begin{itemize}
\item[1)]$\mathcal{E}$ has enough projectives, and each object has a finite projective resolution;
\item[2)] $\mathcal{E} = \mbox{Coh}(X)$ is the category of coherent sheaves on a smooth projective variety $X.$ 
\end{itemize}
\end{theorem}

\begin{proof}  
For any complex 
$$M^{\bullet}: \ldots \to  M^{n-1} \overset{d^{n-1}}\to M^n \overset{d^n}\to M^{n+1} \overset{d^{n+1}}\to \ldots$$
and any integer $n,$ we introduce the following associated truncated complexes:
$$\sigma_{\geq n} M^{\bullet}: \ldots \to 0 \to M^n \overset{d^n}\to M^{n+1} \overset{d^{n+1}}\to \ldots;$$
$$\sigma_{\leq n} M^{\bullet}:= M^{\bullet}/\sigma_{\geq n + 1} M^{\bullet}.$$



These are called ``stupid'' truncations, cf. the proof of lemma \ref{k0inj}. For $n$ being the degree of the maximal non-zero component of $M^{\bullet},$ we get the following conflation in $\mathcal{C}^b(\mathcal{E}):$
$$M^n [-n] \rightarrowtail M^{\bullet} \twoheadrightarrow \sigma_{\geq n - 1} M^{\bullet}.$$
These conflations give rise to the principle of (finite) {\it d\'{e}vissage} (the French word for ``unscrewing''):

\begin{lemma} \label{devis}
For an exact category $\mathcal{B},$ its derived category $\mathcal{D}^b(\mathcal{B})$ is generated (as a triangulated category) by $\mathcal{B}$ concentrated in degree $0.$ All objects of $\mathcal{C}^b(\mathcal{B})$ can be obtained as finitely iterated extensions of stalk complexes.
\end{lemma}

\begin{proposition} \label{comderext}
Suppose that for two bounded complexes $M, N \in \mathcal{C}^b(\mathcal{E}),$ all extensions between their components are trivial:
\begin{equation} \label{exttriv}
\Ext^p_{\mathcal{E}}(M^i, N^j) = 0, \quad \forall p > 0, \forall i, j \in \mathbb{Z}.
\end{equation}
Then the canonical map:
$$\Ext^p_{\mathcal{C}^b(\mathcal{E})} (M, N) \to \Ext^p_{\mathcal{D}^b(\mathcal{E})} (M, N)$$
is bijective for all $p > 0.$
\end{proposition}

\noindent {\it Proof of Proposition \ref{comderext}.\/}
It is a classical fact that we have a family of conflations

$$\Sigma^{-i} M \rightarrowtail C(1_{\Sigma^{-i} M}) \twoheadrightarrow \Sigma^{-(i-1)} M, \quad \quad i > 0,$$
where $C(f)$ is a mapping cone and $\Sigma$ is the shift functor. Thus, we have a complex

\begin{equation} \label{contrresproj}
R(M)^{\bullet} := \ldots \to C(1_{\Sigma^{-3} M}) \to C(1_{\Sigma^{-2} M}) \to C(1_{\Sigma^{-1} M})  \to 0,
\end{equation}
with a quasi-isomorphism $R(M)^{\bullet} \overset{\qis}\to M$ in $\mathcal{C}(\mathcal{C}^b(M)).$
Consider the category of bounded graded objects ${\bf gr^b}(\mathcal{E}),$ whose objects are the $\mathbb{Z}-$graded families of objects of $\mathcal{E}$ and morphisms are given component-wise. There exists a natural forgetful functor 
$$\mathcal{C}^b(\mathcal{E}) \overset{\forget}\to {\bf gr^b}(\mathcal{E}),$$
that is known to be exact and admits an exact left adjoint $\forget_{\lambda}:$
$$\mathcal{C}^b(\mathcal{E}) \overset{\forget_{\lambda}}\leftarrow {\bf gr^b}(\mathcal{E}).$$
It is known as well that 
$$\forget_{\lambda} \circ \forget(X) \overset\sim\to C(1_X),$$
hence we have the following isomorphisms:
\begin{equation} \label{extcones}
\Ext^p_{\mathcal{C}^b(\mathcal{E})}(C(1_X), Y)  \overset\sim\to \Ext^p_{\mathcal{C}^b(\mathcal{E})}(\forget_{\lambda} \circ \forget(X), Y) \overset\sim\to \Ext^p_{{\bf gr^b}(\mathcal{E})}(\forget(X), \forget(Y)).
\end{equation}
Clearly, by (\ref{exttriv}), 
$$\Ext^p_{{\bf gr^b}(\mathcal{E})}(\forget(\Sigma^{-i} M), \forget(N)) = 0,\quad\quad\quad \forall i,$$
hence
$$\Ext^p_{\mathcal{C}^b(\mathcal{E})}(C(1_{\forget(\Sigma^{-i} M}), N) = 0,\quad\quad\quad \forall i.$$

Therefore, the complex $R(M)^{\bullet}$ is actually a $\Hom(?, N)-$acyclic resolution of the complex $M$ in $\mathcal{C}^b(\mathcal{E})$ and can be used to compute $\Ext^p_{\mathcal{C}^b(\mathcal{E})}(M,N).$ Namely, these extensions are exactly the homologies of the complex $\Hom_{\mathcal{C}^b(\mathcal{E})}(R(M)^{\bullet},N),$ i.e. quotients of the set of morphisms $\Hom_{\mathcal{C}^b(\mathcal{E})}({\Sigma^{-(p)} M}, N)$ by the subset of morphisms which factor through $C(1_{\Sigma^{-p} M}).$
This last subset is known to be the subset of null-homotopic morphisms. Therefore, we obtain that extensions are exactly the morphisms in the homotopy category of bounded complexes ${\bf K}^b(\mathcal{E}):$

$$\Ext^p_{\mathcal{C}^b(\mathcal{E})}(M,N) = \Hom_{{\bf K}^b(\mathcal{E})}({\Sigma^{-(p)} M}, N).$$

\begin{lemma} \label{homderext}
Under the condition (\ref{exttriv}), we have a canonical bijection:
$$\Hom_{{\bf K}^b(\mathcal{E})}(M, N) = \Hom_{\mathcal{D}^b(\mathcal{E})}(M, N).$$
\end{lemma}

The lemma uses nothing but the d\'{e}vissage. The statement of the lemma holds, of course, for shifted complexes as well; thus, we get
$$\Ext^p_{\mathcal{C}^b(\mathcal{E})}(M,N) = \Hom_{\mathcal{D}^b(\mathcal{E})}({\Sigma^{-(p)} M}, N).$$
The right-hand side is nothing but $\Hom_{\mathcal{D}^b(\mathcal{E})}({M}, \Sigma^{p}  N),$ that is equal to $\Ext^p_{\mathcal{D}^b(\mathcal{E})}({M},  N).$
{{ $\Box$}\smallskip\par}

\noindent Recall that $\mathcal{P}$ denotes the full subcategory of the projective objects of $\mathcal{E}.$

\begin{corollary} \label{projext}
For all $P \in \mathcal{C}^b(\mathcal{P})$ and all $M \in \mathcal{C}^b(\mathcal{E}),$ we have a canonical bijection:
$$\Ext^p_{\mathcal{C}^b(\mathcal{E})} (P, M) = \Ext^p_{\mathcal{D}^b(\mathcal{E})} (P, M), \quad \forall p > 0.$$
In particular, bounded acyclic complexes of projectives are projective objects in the category $\mathcal{C}^b(\mathcal{P}).$
\end{corollary}

\begin{lemma} [{\cite[4.1,~Lemma,~b)]{Kel1}}]\label{projres}
Assume that $\mathcal{E}$ has enough projectives, and each object of $\mathcal{E}$ has a finite projective  resolution. Then for each $A \in \mathcal{C}^b(\mathcal{E}),$ there exists a deflation 
$$P \stackrel{\qis}{\twoheadrightarrow} A$$
that is a quasi-isomorphism, with $P \in \mathcal{C}^b(\mathcal{P}).$ 
\end{lemma}

Let us return to the proof of Theorem \ref{condd}. For part 1), by lemma \ref{projres}, in condition (C4) we can take deflation quasi-isomorphisms with $L_1, L_2 \in \mathcal{C}^b(\mathcal{P}).$ For those, conditions (i) and (ii) hold by Proposition \ref{projext}. For such $L_1, L_2,$ each $Y$ has projective components as well, thus Proposition \ref{projext} implies condition (iii).

Now, for part 2) we recall the next known facts, the first of which easily follows from Serre's theorem which connects the category $Coh(X)$ and the category of finitely generated modules over the homogeneous coordinate algebra of $X:$ 

\begin{lemma}
\begin{itemize}
\item[(a)] For any $F \in Coh(X), N \in \mathbb{N},$ there exists a deflation 
$E \twoheadrightarrow F,$
where $E$ is of the form
$$E = \bigoplus\limits_{i=1}^n \mathcal{O}(- a_i), \quad\quad\quad a_i > N, \forall i = 1, \ldots, n,$$
for some $n.$
\item[(b)] {\cite[Theorem~III.5.2.b)]{Har}}
For every $F \in Coh(X),$ there exists a nonnegative integer $n_0(F),$ such that for each $n > n_0(F),$ we have:
$$\Ext^p_{Coh(X)} (\mathcal{O}(- n), F)  = H^p(X, F(n)) = 0, \quad\quad\quad \forall p > 0.$$
\end{itemize}
\end{lemma}

Objects of the above form $E$ generate an additive category, namely the category $\mathcal{V}_N(X)$ of vector bundles over $X$ which are sums of line bundles $\mathcal{O}(-j), j \geq N.$



Let us return to the proof of part 2) of Theorem \ref{condd}. Similarly to Lemma \ref{projres} (cf. also \cite[Theorem~12.7]{Buh}), we can show that for each $N > 0$ and for each $A \in \mathcal{C}^b(Coh(X)),$ there exists a deflation quasi-isomorphism $B \stackrel{\qis}{\twoheadrightarrow} A,$ with $B \in \mathcal{C}^b(\mathcal{V}_N(X)).$ Now we take $N_2$ to be the maximum among the numbers $n_0(C^j),$ where $C^j$ are among the (finitely many!) non-zero components of $C,$ and find $L_2 \in \mathcal{V}_{N_2}(X).$ Similarly, we take $N_1$ to be the maximum among the numbers $n_0(B^j)$ and $n_0(L_2^j),$ where $B^j$ and $L_2^j$ are among the finitely many non-zero components of $B,$ respectively of $L_2.$ Now we take $L_1 \in \mathcal{V}_{N_1}(X).$ As in part 1), for such $L_1, L_2$ all the conditions (i)-(iii) hold, by Proposition \ref{comderext}.
\end{proof}


\section{Euler form and quantum tori}

Let $\mathcal{C}_{ac}^b(\mathcal{E})$ denote the category of bounded acyclic complexes over $\mathcal{E}.$

\begin{definition}
The {\it multiplicative Euler form} 

$$\left\langle \cdot, \cdot \right\rangle: K_0 (\mathcal{C}_{ac}^b(\mathcal{E})) \times  K_0 (\mathcal{C}^b(\mathcal{E})) \to \mathbb{Q}^{\ast}$$
is given by the alternating product: 
$$\left\langle K, A \right\rangle := \prod\limits_{p=0}^{+\infty} |\Ext^{p}_{\mathcal{C}^b(\mathcal{E})} (K, A)|^{(-1)^p}.$$
Using a similar product we define 
$$\left\langle \cdot, \cdot \right\rangle: K_0 (\mathcal{C}^b(\mathcal{E})) \times  K_0 (\mathcal{C}_{ac}^b(\mathcal{E})) \to \mathbb{Q}^{\ast}.$$
These two forms clearly coincide on $K_0 (\mathcal{C}_{ac}^b(\mathcal{E})) \times  K_0 (\mathcal{C}_{ac}^b(\mathcal{E})),$ so it is harmless to denote them by the same symbol.
\end{definition}

By the following lemma, this alternating product is well-defined, i.e. all but a finite number of factors equal 1. Therefore, by the five-lemma, it is bilinear. In fact, in our framework, one can define this form on $K_0(\mathcal{C}^b(\mathcal{E})) \times  K_0 (\mathcal{C}^b(\mathcal{E}));$ but in the ${\mathbb{Z}/2}-$graded case, this will not be true any more.

\begin{lemma} \label{homfin}
For each pair of objects $A, B \in \mathcal{C}^b(\mathcal{E}),$ the property of local homological finiteness holds, i.e. there exists $n_0 > 0$ such that for all $p > n_0,$ we have
$$\Ext^p_{\mathcal{C}^b(\mathcal{E})} (A, B) = 0.$$
\end{lemma}

\begin{proof}
Using the five-lemma and the stupid truncations (d\'{e}vissage) for both arguments, we observe that it is enough to prove the local homological finiteness only for stalk complexes $A$ and $B.$ For each $A,$ we have a conflation
$$\Sigma^{-1} A \rightarrowtail C(1_{\Sigma^{-1} A}) \twoheadrightarrow A.$$
By condition (C3), the desired local homological finiteness holds for the graded objects. 
Thus, by (\ref{extcones}), it holds for $C(1_{\Sigma^{-1} A})$ and $B,$ hence it holds for $A, B$ if only if it holds for $\Sigma^{-1} A, B.$ By induction, we get that it is enough to prove the statement for stalk complexes $A, B$ both concentrated in the same degree. For those, the extensions are extensions in $\mathcal{E}$ of their non-zero components, and the condition (C3) implies the desired statement.
\end{proof}

Consider the set $\Iso(\mathcal{C}_{ac}^b(\mathcal{E}))$ of isomorphism classes $[K]$ of bounded acyclic complexes and its quotient by the following set of relations:  
$$\left\langle [K_2] = [K_1 \oplus K_3] | K_1 \rightarrowtail K_2 \twoheadrightarrow K_3 \quad \mbox{is a conflation} \right\rangle.$$ 
If we endow $\Iso(\mathcal{C}_{ac}^b(\mathcal{E}))$ with the addition given by direct sums, this quotient gives the Grothendieck monoid $M_0(\mathcal{C}_{ac}^b(\mathcal{E}))$ of the exact category $\mathcal{C}_{ac}^b(\mathcal{E}).$ We define the {\it quantum affine space of acyclic complexes} $\mathbb{A}_{ac}(\mathcal{E})$ as the $\mathbb{Q}-$vector space generated by elements of $M_0(\mathcal{C}_{ac}^b(\mathcal{E})).$ We endow it with the bilinear multiplication defined below.

\begin{definition}
For $K_1, K_2 \in \mathcal{C}_{ac}^b(\mathcal{E}),$ we define their product as
$$[K_1] \diamond [K_2] := \frac{1}{\left\langle K_1, K_2 \right\rangle} [K_1 \oplus K_2].$$
\end{definition}

By Lemma \ref{homfin}, this product is well-defined; moreover, it is clearly associative. We see that this ring has the class of the zero complex $[0]$ as the unit. Moreover, it is clear that the set of all elements of the form $[K]$ satisfies the Ore conditions. This means that we can make all of them invertible and consider the {\it quantum torus of acyclic complexes} $\mathbb{T}_{ac}(\mathcal{E}).$ It is generated by classes $[K]$ and their inverses $[K]^{-1}.$ Here are simple relations concerning the product.

\begin{lemma} \label{acinvcomm}
For $K_1, K_2 \in \mathcal{C}_{ac}^b(\mathcal{E}),$ we have 
$$[K_1]^{-1} \diamond [K_2]^{-1} = \left\langle K_2, K_1 \right\rangle [K_1 \oplus K_2]^{-1}; \quad [K_1]^{-1} \diamond [K_2] = \frac{\left\langle K_1, K_2 \right\rangle}{\left\langle K_2, K_1 \right\rangle} [K_2] \diamond [K_1]^{-1}.$$
\end{lemma}

\begin{proof}
The first identity is trivial, let us prove the second one.
$$[K_1] \diamond [K_2] = \frac{1}{\left\langle K_1, K_2 \right\rangle} [K_1 \oplus K_2] \Rightarrow [K_2] = \frac{1}{\left\langle K_1, K_2 \right\rangle} [K_1]^{-1} \diamond [K_1 \oplus K_2] \Rightarrow$$
$$\Rightarrow [K_2] \diamond [K_1]^{-1} =  \frac{1}{\left\langle K_1, K_2 \right\rangle} [K_1]^{-1} \diamond [K_1 \oplus K_2] \diamond [K_1]^{-1}.$$
Similarly, we find that
$$[K_1]^{-1} \diamond [K_2] = \frac{1}{\left\langle K_2, K_1 \right\rangle} [K_1]^{-1} \diamond [K_1 \oplus K_2] \diamond [K_1]^{-1},$$
and the desired equality follows.
\end{proof}


Another way to define $\mathbb{T}_{ac}(\mathcal{E})$ is provided by the following simple statement.

\begin{lemma}
The quantum torus of acyclic complexes is isomorphic to the quantum torus of the Grothendieck group $K_0(\mathcal{C}^b_{ac}(\mathcal{E})),$ twisted by the inverse of the Euler form:
$$\mathbb{T}_{ac}(\mathcal{E}) \overset\sim\to \mathbb{T}(K_0(\mathcal{C}^b_{ac}(\mathcal{E})), {\left\langle \cdot, \cdot \right\rangle}^{-1}).$$
In other words, it is the $\mathbb{Q}-$group algebra of $K_0(\mathcal{C}^b_{ac}(\mathcal{E})),$ with the multiplication twisted by the inverse of the Euler form.
\end{lemma}

\begin{proof}
By the construction, $\mathbb{A}_{ac}(\mathcal{E})$ is the monoid algebra of $M_0(\mathcal{C}^b_{ac}(\mathcal{E})),$ with the multiplication, twisted by the Euler form. Therefore, for the proof it is enough to verify that the universal properties of the monoid algebra $\mathbb{Q}[M],$ localized at the elements of the monoid, and of the group algebra $\mathbb{Q}[G]$ of the group $G$, corresponding to this monoid $M$ by the Grothendieck construction, are the same. Indeed, for each $\mathbb{Q}-$algebra $A,$ we have
$$\Hom_{\mathbb{Q}-{\bf{Alg}}}(\mathbb{Q}[G], A) = \Hom_{\bf{Grp}}(G, A^\times) = \Hom_{\bf{Mon}}(M, A^\times);$$
$$\Hom_{\mathbb{Q}-{\bf{Alg}}}(\mathbb{Q}[M][[m]^{-1}, m \in M], A) = \left\{f \in \Hom_{\mathbb{Q}-{\bf{Alg}}}(\mathbb{Q}[M], A) | f(M) \subset A^\times\right\} = $$
\smallskip\smallskip\smallskip\smallskip$\qquad\qquad\qquad\qquad\qquad\qquad\qquad\quad = \Hom_{\bf{Mon}}(M, A^\times).$
\end{proof}

\section{Definition of the semi-derived Hall algebras}
\subsection{Module structure over the quantum torus}
Let us consider the vector space $\mathcal{M}_1(\mathcal{E})$ over $\mathbb{Q}$ whose basis is formed by the isomorphism  classes $[M],$ where $M \in \mathcal{C}^{b}(\mathcal{E}).$ On $\mathcal{M}_1(\mathcal{E}),$ let us define a multiplication by classes of acyclic complexes, generalizing the one from the previous section. Namely, for $K$ acyclic, $M$ arbitrary, we define their products as follows:
$$[K] \diamond [M] = \frac{1}{\left\langle K, M \right\rangle} [K \oplus M]; \quad [M] \diamond [K] = \frac{1}{\left\langle M, K \right\rangle} [K \oplus M].$$
We get a bimodule over $\mathbb{A}_{ac}(\mathcal{E}),$ let us call it $\mathcal{M}_1^{'}(\mathcal{E}).$ 

We quotient $M_1(\mathcal{E})$ by the set of relations
\begin{equation} \label{Zrelations}
\left\langle [L] = [K \oplus M] | K \rightarrowtail L \twoheadrightarrow M \quad \mbox{is a conflation}, K \in \mathcal{C}_{ac}^{b}(\mathcal{E}) \right\rangle,
\end{equation}
to obtain the space $\mathcal{M}_2(\mathcal{E}).$
We will denote classes after the factorization by the same symbols $[M].$ Of course, one gets the same vector space if one starts with the quotient of $\Iso(\mathcal{C}^{b}(\mathcal{E}))$ by this set of relations and then considers the vector space on this basis. It is easy to check that these relations respect the bimodule structure. Thus, $\mathcal{M}_2(\mathcal{E})$ has the induced bimodule structure, let us denote it by $\mathcal{M}_2^{'}(\mathcal{E}).$
By taking the tensor product with the quantum torus, we get a bimodule $\mathcal{M}(\mathcal{E}) := \mathbb{T}_{ac}(\mathcal{E}) \otimes_{\mathbb{A}_{ac}(\mathcal{E})} \mathcal{M}_2^{'}(\mathcal{E}) \cong  \mathcal{M}_2^{'}(\mathcal{E}) \otimes_{\mathbb{A}_{ac}(\mathcal{E})} \mathbb{T}_{ac}(\mathcal{E})$ over $\mathbb{T}_{ac}(\mathcal{E}).$

\subsection{Freeness over the quantum torus}
\begin{theorem} \label{zfree}
$\mathcal{M}(\mathcal{E})$ is free as a right module over $\mathbb{T}_{ac}(\mathcal{E}).$  Each choice of representatives of the quasi-isomorphism classes in $\mathcal{C}^b(\mathcal{E})$ yields a basis.
\end{theorem}

\begin{proof}
Assume that two complexes $M$ and $M'$ are quasi-isomorphic to each other. This means that there is a sequence of objects $M_0 = M, M_1, M_2,\ldots, M_n = M',$ such that for each $i = 1, 2, \ldots, n$ there is either a conflation 
$$K \rightarrowtail M_{i-1} \stackrel{\qis}{\twoheadrightarrow} M_i,$$
or a conflation
$$K \rightarrowtail M_i \stackrel{\qis}{\twoheadrightarrow} M_{i-1},$$
with $K$ acyclic. Therefore, we either have
$$[M_i] = [K \oplus M_{i-1}] = \left\langle K, M_{i-1} \right\rangle [K] \diamond [M_{i-1}],$$
or
$$[M_{i-1}] = [K \oplus M_i] = \left\langle K, M_i \right\rangle [K] \diamond [M_i] \Rightarrow [M_i] = \frac{1}{\left\langle K, M_i \right\rangle} [K]^{-1} \diamond [M_{i-1}].$$
It follows that $[M'] \in \mathbb{T}_{ac}(\mathcal{E}) \diamond [M].$ Therefore, the quasi-isomorphism classes of complexes generate $\mathcal{M}(\mathcal{E})$ over $\mathbb{T}_{ac}(\mathcal{E}).$ It remains to prove that they are independent over this quantum torus.


Since each relation in $\mathcal{M}_1(\mathcal{E})$ from the definition of the underlying vector space of $\mathcal{M(E)}$ identifies two elements in the same quasi-isomorphism class, one can decompose $\mathcal{M(E)}$ into the direct sum 
$$\mathcal{M(E)} = \bigoplus\limits_{\alpha \in \Iso(\mathcal{D}^b(\mathcal{E}))} \mathcal{M}_{\alpha} \mathcal{(E)},$$
where $\mathcal{M}_{\alpha} \mathcal{(E)}$ is the component containing the classes of all objects whose isomorphism class in $\mathcal{D}^b(\mathcal{E})$ is $\alpha.$ We claim that for each $\alpha,$ the $\mathbb{T}_{ac}-$submodule $\mathcal{M}_{\alpha}(\mathcal{E})$ is free of rank one. Let $M$ be an object of $\mathcal{C}^b(\mathcal{E}).$ By the above argument, the map
$$\mathbb{T}_{ac}(\mathcal{E}) \to \mathcal{M}_{[M]}(\mathcal{E}), \quad [K] \mapsto [K] \diamond [M]$$
is surjective. Since $\mathbb{T}_{ac}(\mathcal{E})$ is the (twisted) group algebra of $K_0(\mathcal{C}^b_{ac}(\mathcal{E})),$ the following Lemma \ref{k0inj} shows that its composition with the natural map 
$$\mathcal{M}_{[M]}(\mathcal{E}) \to \mathcal{M}(\mathcal{E}) \to \mathbb{Q}[K_0(\mathcal{C}^b(\mathcal{E}))]$$
is injective. Therefore, it is bijective. This completes the proof.
\end{proof}




\begin{lemma} \label{k0inj}
The natural map 
$$i: K_0(C_{ac}^b(\mathcal{E})) \rightarrow K_0 (C^b(\mathcal{E})), [M] \mapsto [M]$$
is injective.
\end{lemma}

\begin{proof}
Thanks to the d\'{e}vissage principle (using stupid truncations), $K_0 (C^b(\mathcal{E}))$ is spanned by stalk complexes. Moreover, it is isomorphic to the coproduct of $\mathbb{Z}$ copies of $K_0 (\mathcal{E}):$ the  following two homomorphisms $f, g$
$$f: K_0 (C^b(\mathcal{E})) \to \coprod\limits_{\mathbb{Z}} K_0 (\mathcal{E}), [\ldots \to M^{n-1} \to M^n \to M^{n+1} \to \ldots] \mapsto (\ldots, [M^{n-1}], [M^n], [M^{n+1}],\ldots),$$
$$g: \coprod\limits_{\mathbb{Z}} K_0 (\mathcal{E}) \to K_0 (\mathcal{C}^b(\mathcal{E})), (\ldots, [M^{n-1}], [M^n], [M^{n+1}],\ldots) \mapsto [\ldots \to M^{n-1} \to M^n \to M^{n+1} \to \ldots]$$
are clearly inverse to each other. 
For acyclic bounded complexes, we also have the so-called ``intelligent truncations''. Recall that a complex 
$$K = \ldots \to K^{n-1} \overset{d^{n-1}}{\to} K^n \overset{d^n}{\to} K^{n+1} \to \ldots$$
over an exact category $\mathcal{E}$ is acyclic if each of its differentials $d^n$ factors into a composition $K^n~\overset{\pi^n}{\to}~Z^n~\overset{i^n}{\to}~K^{n+1},$ where all $\pi^n$ are deflations, all $i^n$ are inflations, and for each $n,$ the sequence $Z^{n-1}~\overset{i^{n-1}}{\to}~K^n~\overset{\pi^n}{\to}~Z^n$ is a conflation in $\mathcal{E}.$ Assume that $K$ is bounded and $m$ is the degree of the rightmost non-zero component. Then $Z^k = 0,$ for $k \geq m,$ hence $i^{n-1}: Z^{n-1} \to K^n$ is an isomorphism. It follows that we have a conflation in $\mathcal{C}^b(\mathcal{E})$ and, more precisely, in $\mathcal{C}^b_{ac}(\mathcal{E}):$

\begin{equation} \label{intelligenttrunc}
\xymatrix@R=0.6cm{
*+++{\ldots} \ar@{->}[r] &*+++{K^{m-2}} \ar@{=}[d] \ar@{->>}[r]^{\pi^{m-2}} &*+++{Z^{m-2}} \ar@{>->}[d]_{i^{m-2}} \ar@{->>}[r] &*+++{0} \ar@{>->}[d] \ar@{->}[r] &*+++{0} \ar@{>->}[d] \ar@{->}[r] &*+++{\ldots}\\
*+++{\ldots} \ar@{->}[r] &*+++{K^{m-2}} \ar@{->>}[d] \ar@{->}[r]^{d^{m-2}} & *+++{K^{m-1}} \ar@{->>}[d]_{\pi^{m-1}} \ar@{->}[r]^{d^{m-1}} & *+++{K^m} \ar@{=}[d] \ar@{->}[r]  & *+++{0} \ar@{=}[d] \ar@{->}[r] & *+++{\ldots}\\
*+++{\ldots} \ar@{->}[r] & *+++{0} \ar@{->}[r] & *+++{Z^{m-1}} \ar@{=}[r]^{i^{m-1}} & *+++{K^m} \ar@{->}[r] & *+++{0} \ar@{->}[r] & *+++{\ldots}\\
}
\end{equation}

The top complex in the conflation (\ref{intelligenttrunc}) is denoted by $\tau_{<(m-1)},$ the bottom one is denoted by $\tau_{\geq (m-1)}.$ The intelligent truncations show, by induction, that any acyclic bounded complex is an iterated extension of complexes of the form 
$$\ldots \to 0 \to X = X \to 0 \to \ldots,$$
i.e. of cones of identities of stalk complexes. Thus, the Grothendieck group $K_0(\mathcal{C}^b_{ac}(\mathcal{E}))$ is also isomorphic to the coproduct of $\mathbb{Z}$ copies of $K_0 (\mathcal{E}):$ consider the homomorphism
$$h: \coprod\limits_{\mathbb{Z}} K_0 (\mathcal{E}) \to K_0 (\mathcal{C}^b_{ac}(\mathcal{E})),  \quad (\ldots, [0], [X], [0], \ldots) \mapsto \ldots \to 0 \to X = X \to 0 \to \ldots,$$ 
where in the argument, $[X]$ is in the $n$th component, in the image, $X$ is in $(n-1)$th and $n$th components. This homomorphism is clearly invertible.

Now to prove injectivity of $i,$ it is enough to prove the injectivity of the following composed map:
$$\coprod\limits_{\mathbb{Z}} K_0 (\mathcal{E}) \overset{h}{\overset{\sim}{\to}}  K_0 (\mathcal{C}^b_{ac}(\mathcal{E})) \overset{i}{\to} K_0 (\mathcal{C}^b(\mathcal{E})) \overset{f}{\overset{\sim}{\to}} \coprod\limits_{\mathbb{Z}} K_0 (\mathcal{E}).$$
We will do it by contradiction. Assume that the element  
$${\bf X} = (\ldots, \sum\limits_{i=1}^{k_n} x_i [X_i],  \ldots)$$
of $\coprod\limits_{\mathbb{Z}} K_0 (\mathcal{E})$  (the described term sits in $n-$th component), is nonzero and belongs to the kernel of $f \circ i \circ h.$ Consider the maximal $m$ such that the $m$th component of ${\bf X}$ is nonzero. Following the definitions of the homomorphisms $f, i, h,$ one observes that then the  $m$th components of ${\bf X}$ and of $f(i(h({\bf X})))$ coincide, thus the latter is also non-zero. Contradiction.
\end{proof}

\subsection{Multiplication}
Consider a pair of complexes $A, B \in \mathcal{C}^b(\mathcal{E}).$ For each class $\varepsilon$ in the extension group $\Ext^1_{\mathcal{C}^b(\mathcal{E})}(A, B)$ represented by a conflation $B \rightarrowtail E \twoheadrightarrow A,$ we denote by $\mt(\varepsilon)$ the isomorphism class of $E$ in $\mathcal{C}^b(\mathcal{E}).$ It is well-defined, i.e. it does not depend on the choice of the representative of $\varepsilon.$ We also consider its class $[\mt(\varepsilon)]$ in $\mathcal{M}(\mathcal{E}).$

\begin{definition}
We define a $\mathbb{Q}-$bilinear map 
$$\diamond : \mathcal{M}_1(\mathcal{E}) \times \mathcal{M}_1(\mathcal{E}) \to \mathcal{M}(\mathcal{E})$$ 
by the following rule:
\begin{equation} \label{diamonddef1}
[L] \diamond [M] = \frac{1}{\left\langle K, L \right\rangle} [K]^{-1} \diamond \sum\limits_{E \in \Iso (\mathcal{C}^b(\mathcal{E}))} (\frac{|\Ext^1_{\mathcal{C}^b(\mathcal{E})}(L', M)_E|}{|\Hom(L', M)|} [E]),
\end{equation}
or, equivalently, as
\begin{equation} \label{diamonddef2}
[L] \diamond [M] = \frac{1}{\left\langle K, L \right\rangle} [K]^{-1} \diamond \sum\limits_{\varepsilon' \in \Ext^1_{\mathcal{C}^b(\mathcal{E})}(L',M)} \frac{[\mt(\varepsilon')]}{|\Hom(L', M)|},
\end{equation}
where 
$K \rightarrowtail L' \stackrel{\qis}{\twoheadrightarrow} L$ 
is a conflation such that $L'$ satisfies
\begin{equation} \label{extpmult}
\Ext^p_{\mathcal{C}^b(\mathcal{E})} (L', M) = \Ext^p_{\mathcal{D}^b(\mathcal{E})} (L', M), \quad \forall p > 0.
\end{equation} 
\end{definition}

By part (i) of assumption (C4), at least one such $L'$ always exists. We want to prove that the definition is correct and compatible with the bimodule structure over $\mathbb{A}_{ac}(\mathcal{E})$ and that it descends to $\mathcal{M}(\mathcal{E}).$

\begin{proposition} \label{welldef}
The map $\diamond$ is well-defined, i.e. it does not depend on the choice of a conflation
$K \rightarrowtail L' \stackrel{\qis}{\twoheadrightarrow} L $
such that the isomorphism (\ref{extpmult}) holds.
\end{proposition}

\begin{proof}
Suppose we have two conflations whose deflations are quasi-isomorphisms
$$K_1 \rightarrowtail L_{1}' \stackrel{\qis}{\twoheadrightarrow} L, K \rightarrowtail L' \stackrel{\qis}{\twoheadrightarrow} L,$$
and isomorphism (\ref{extpmult}) holds for both $L'$ and $L_{1}'.$ Then for the pullback $L_{2}'$ of the two deflations arising in these conflations, we also have deflations which are quasi-isomorphisms
$L_{2}' \stackrel{\qis}{\twoheadrightarrow} L', L_{2}' \stackrel{\qis}{\twoheadrightarrow} L_{1}'.$
By assumption (C4), there exists a deflation quasi-isomorphism
$L'' \stackrel{\qis}{\twoheadrightarrow} L_{2}',$
with $L''$ satisfying the condition (\ref{extpmult}) for the extensions by $M.$ Since the set of deflations is closed under composition, we have a deflation 
$L'' \stackrel{\qis}{\twoheadrightarrow} L.$
Completing all these deflations to conflations and using the axiom of the exact category concerning push-outs, we can obtain the following commutative diagram with $H, N, K$ acyclic:
\begin{equation} \label{diag1}
\xymatrix@R=0.6cm{
*+++{H} \ar@{>->}[d] \ar@{=}[r] &*+++{H} \ar@{>->}[d] \ar@{->>}[r] & *+++{0} \ar@{>->}[d]\\
*+++{N} \ar@{->>}[d] \ar@{>->}[r]  & *+++{L''} \ar@{->>}[d] \ar@{->>}[r]  & *+++{L} \ar@{=}[d] \\
*+++{K}  \ar@{>->}[r]  & *+++{L'} \ar@{->>}[r]  & *+++{L} \\
}
\end{equation}


Therefore, it is sufficient to prove that for an arbitrary pair of conflations 
$$K \rightarrowtail L' \stackrel{\qis}{\twoheadrightarrow} L, N \rightarrowtail L'' \stackrel{\qis}{\twoheadrightarrow} L$$
satisfying the condition (\ref{extpmult}) and such that there exists a commutative diagram (\ref{diag1}), we have the identity
\begin{equation} \label{f1}
\frac{1}{\left\langle K, L \right\rangle} [K]^{-1} \diamond \sum\limits_{\varepsilon' \in \Ext^1_{\mathcal{C}^b(\mathcal{E})}(L',M)} \frac{[\mt(\varepsilon')]}{|\Hom(L', M)|}  =
 \frac{1}{\left\langle N, L \right\rangle} [N]^{-1} \diamond \sum\limits_{\varepsilon'' \in \Ext^1_{\mathcal{C}^b(\mathcal{E})}(L'',M)} \frac{[\mt(\varepsilon'')]}{|\Hom(L'', M)|}.
\end{equation}
Using the conflation
$H \rightarrowtail N \twoheadrightarrow K,$
we find out that 
$$[N] = [K \oplus H] =  \left\langle H, K \right\rangle [H] \diamond [K] \quad \Leftrightarrow \quad [N] \diamond [K]^{-1} = \left\langle H, K \right\rangle [H] \quad \Leftrightarrow $$
$$\Leftrightarrow \quad  [K]^{-1} = \left\langle H, K \right\rangle [N]^{-1} \diamond [H].$$
Therefore, the right hand side of (\ref{f1}) is equal to
$$\frac{\left\langle H, K \right\rangle}{\left\langle K, L \right\rangle} [N]^{-1} \diamond [H] \diamond  \sum\limits_{\varepsilon' \in \Ext^1_{\mathcal{C}^b(\mathcal{E})}(L',M)} \frac{[\mt(\varepsilon')]}{|\Hom(L', M)|} =$$
$$= \frac{\left\langle H, L' \right\rangle}{\left\langle K, L \right\rangle \left\langle H, L \right\rangle} [N]^{-1} \diamond [H] \diamond \sum\limits_{\varepsilon' \in \Ext^1_{\mathcal{C}^b(\mathcal{E})}(L',M)} \frac{[\mt(\varepsilon')]}{|\Hom(L', M)|}  =$$
\begin{equation} \label{f2}
= \frac{\left\langle H, L' \right\rangle}{\left\langle N, L \right\rangle} [N]^{-1} \diamond [H] \diamond  \sum\limits_{\varepsilon' \in  \Ext^1_{\mathcal{C}^b(\mathcal{E})}(L',M)} \frac{[\mt(\varepsilon')]}{|\Hom(L', M)|}.
\end{equation}

From condition (\ref{extpmult}) for $i=1$ for $L'$ and $L''$ which are quasi-isomorphic to each other, we see that the morphism 
$\Ext^1_{\mathcal{C}^b(\mathcal{E})} (L', M) \to \Ext^1_{\mathcal{C}^b(\mathcal{E})} (L'', M)$
arising from the long exact sequence of extensions of elements of the conflation $N \rightarrowtail L'' \stackrel{\qis}{\twoheadrightarrow} L'$ by $M$ is in fact an isomorphism. That means that for each class $\varepsilon'$ in $\Ext^1_{\mathcal{C}^b(\mathcal{E})} (L', M)$ represented by a conflation
$$\quad M  \rightarrowtail  E'  \twoheadrightarrow L',$$
there is a unique class $\varepsilon''$ in $\Ext^1_{\mathcal{C}^b(\mathcal{E})} (L', M)$ represented by a conflation
$$M \rightarrowtail E'' \twoheadrightarrow L''$$
such that the following diagram commutes:
\begin{equation} \label{diag2}
\xymatrix@R=0.6cm{
*+++{0} \ar@{>->}[d] \ar@{>->}[r] &*+++{H} \ar@{>->}[d] \ar@{=}[r] & *+++{H} \ar@{>->}[d]\\
*+++{M} \ar@{=}[d] \ar@{>->}[r]  & *+++{E''} \ar@{->>}[d] \ar@{->>}[r]  & *+++{L''} \ar@{->>}[d] \\
*+++{M}  \ar@{>->}[r]  & *+++{E'} \ar@{->>}[r]  & *+++{L'.} \\
}
\end{equation}

Using the conflation
$$H  \rightarrowtail   E' \twoheadrightarrow E'',$$
we see that formula (\ref{f2}) can be rewritten as follows:
$$\frac{\left\langle H, L' \right\rangle}{\left\langle N, L \right\rangle} [N]^{-1} \diamond  \sum\limits_{\varepsilon'' \in \Ext^1_{\mathcal{C}^b(\mathcal{E})}(L'',M)} \frac{[\mt(\varepsilon'')]}{|\Hom(L', M)| \left\langle H, L' \right\rangle \left\langle H, M \right\rangle} =$$
$$= \frac{1}{\left\langle N, L \right\rangle \left\langle H, M \right\rangle} [N]^{-1} \diamond  \sum\limits_{\varepsilon'' \in \Ext^1_{\mathcal{C}^b(\mathcal{E})}(L'',M)} \frac{[\mt(\varepsilon'')]}{|\Hom(L', M)|} =$$
$$= \frac{1}{\left\langle N, L \right\rangle} [N]^{-1}\diamond  \sum\limits_{\varepsilon'' \in \Ext^1_{\mathcal{C}^b(\mathcal{E})}(L'',M)} \frac{[\mt(\varepsilon'')]}{|\Hom(L'', M)|},$$
as desired. In the last equation, we use nothing more than 
$$\left\langle H, M \right\rangle = \frac{|\Hom(L'', M)|}{|\Hom(L', M)|}.$$
This identity holds, since, by assumption, in the long exact sequence of extensions of elements of the conflation $H \rightarrowtail L'' \stackrel{\qis}{\twoheadrightarrow} L'$ by $M$ all morphisms 
$$\Ext^p_{\mathcal{C}^b(\mathcal{E})} (L', M) \to \Ext^p_{\mathcal{C}^b(\mathcal{E})} (L'', M)$$
are isomorphisms. This completes the proof.
\end{proof}


\begin{proposition} \label{zprodcomp}
The products $K \diamond L$ and $L \diamond K,$ given in this section coincide with those given in section 5.1, for an acyclic complex $K$ and an arbitrary $L.$ 
\end{proposition}

\begin{proof}
Denote the map defined in this section by $\diamond_1,$ the one used in the definition of $\mathbb{T}_{ac}(\mathcal{E})-$action by $\diamond_2.$ To compute the $K \diamond_1 L,$ we consider some conflation
$K'' \rightarrowtail K' \stackrel{\qis}{\twoheadrightarrow} K$
such that 
$$\Ext^p_{\mathcal{C}^b(\mathcal{E})}(K', L) \cong \Ext^p_{\mathcal{D}^b(\mathcal{E})}(K', L), \forall p > 0.$$
Since $K$ is acyclic, so is $K'.$ Thus, we find that 
$\Ext^p_{\mathcal{C}^b(\mathcal{E})}(K', L) = 0, \forall p > 0.$
Hence 
$$K \diamond_1 L = \frac{1}{\left\langle K'', K \right\rangle} [K'']^{-1} \diamond_2 \frac{1}{|\Hom(K', L)|} [K' \oplus L] =$$
$$= \frac{1}{\left\langle K'', K \right\rangle \left\langle K', L \right\rangle} [K'']^{-1} \diamond_2 [K' \oplus L] =$$
\begin{equation} \label{k''}
= \frac{1}{\left\langle K'', K \right\rangle \left\langle K'', L \right\rangle \left\langle K, L \right\rangle} [K'']^{-1} \diamond_2 [K' \oplus L].
\end{equation}
Using the conflation 
$K'' \rightarrowtail K' \oplus L \stackrel{\qis}{\twoheadrightarrow} K \oplus L,$
we get
$$[K''] \diamond_2 [K \oplus L] = \frac{1}{\left\langle K'', K \oplus L \right\rangle} [K' \oplus L].$$
Using this, we rewrite the expression (\ref{k''}) as
$\frac{1}{\left\langle K, L \right\rangle} [K \oplus L],$
which equals $K \diamond_2 L$ by definition. For $L \diamond K,$ the proof is similar.
\end{proof}

\begin{proposition}
The map $\diamond$ is compatible with the bimodule structure on $\mathcal{M}_1^{'}(\mathcal{E})$ and with relations (\ref{Zrelations}).
\end{proposition}

\begin{proof}
It is easy to check that $\diamond$ is simultaneously compatible with the bimodule structure and with the relations if and only if for any conflations $K \rightarrowtail L \twoheadrightarrow M, N \rightarrowtail P \twoheadrightarrow Q$ with $K, Q \in \mathcal{C}_{ac}^b(\mathcal{E}),$ we have
\begin{equation} \label{compat}
\frac{1}{\left\langle K, M \right\rangle} [L] \diamond [N] = [K] \diamond ([M] \diamond [N]); \frac{1}{\left\langle N, Q \right\rangle} [M] \diamond [P] = ([M] \diamond [N]) \diamond [Q].
\end{equation}
Thus, it is enough to check the identities (\ref{compat}). Using the pull-back as in Proposition (\ref{welldef}), we can construct the following diagram:

\begin{equation} 
\xymatrix@R=0.6cm{
*+++{K''} \ar@{>->}[d] \ar@{>->}[r] &*+++{L''} \ar@{>->}[d] \ar@{=}[r] & *+++{M''} \ar@{>->}[d]\\
*+++{K'} \ar@{->>}[d] \ar@{>->}[r]  & *+++{L'} \ar@{->>}[d] \ar@{->>}[r]  & *+++{M'} \ar@{->>}[d] \\
*+++{K}  \ar@{>->}[r]  & *+++{L} \ar@{->>}[r]  & *+++{M,} \\
}
\end{equation}
where all elements in the first row and the first column are acyclic, and the last two columns can be used to define $L \diamond N$ and $M \diamond N.$ We have 
\begin{equation}
\frac{1}{\left\langle K, M \right\rangle} [L] \diamond [N] = \frac{1}{\left\langle K, M \right\rangle\left\langle L'', L \right\rangle} [L'']^{-1} \diamond \sum\limits_{\varepsilon_L \in \Ext^1_{\mathcal{C}^b(\mathcal{E})}(L', N)} \frac{[\mt(\varepsilon_L)]}{|\Hom(L', N)|}.
\end{equation}
Using the same arguments as in the proof of Proposition \ref{welldef} and the diagram similar to (\ref{diag2}), we can rewrite this as follows:
$$\frac{\left\langle K', N\right\rangle\left\langle K', M'\right\rangle}{\left\langle K, M \right\rangle\left\langle L'', L \right\rangle} [L'']^{-1}  \diamond [K'] \diamond \sum\limits_{\varepsilon_M \in \Ext^1_{\mathcal{C}^b(\mathcal{E})}(M', N)} \frac{[\mt(\varepsilon_M)]}{|\Hom(L', N)|} = $$
$$= \frac{\left\langle K', N\right\rangle\left\langle K, M''\right\rangle\left\langle K'', M''\right\rangle}{\left\langle K'', K \right\rangle\left\langle M'', K \right\rangle\left\langle M'', M \right\rangle} [L'']^{-1} \diamond [K'] \diamond \sum\limits_{\varepsilon_M \in \Ext^1_{\mathcal{C}^b(\mathcal{E})}(M', N)} \frac{[\mt(\varepsilon_M)]}{|\Hom(L', N)|} = $$
$$= \frac{\left\langle K', N\right\rangle\left\langle K, M''\right\rangle}{\left\langle M'', K \right\rangle\left\langle M'', M \right\rangle} [M'']^{-1} \diamond [K] \diamond \sum\limits_{\varepsilon_M \in \Ext^1_{\mathcal{C}^b(\mathcal{E})}(M', N)} \frac{[\mt(\varepsilon_M)]}{|\Hom(L', N)|} = $$
$$= \frac{\left\langle K', N\right\rangle}{\left\langle M'', M \right\rangle} [K] \diamond [M'']^{-1} \diamond \sum\limits_{\varepsilon_M \in \Ext^1_{\mathcal{C}^b(\mathcal{E})}(M', N)} \frac{[\mt(\varepsilon_M)]}{|\Hom(L', N)|} = $$
$$= [K] \diamond \left(\frac{1}{\left\langle M'', M\right\rangle} [M'']^{-1} \diamond \sum\limits_{\varepsilon_M \in \Ext^1_{\mathcal{C}^b(\mathcal{E})}(M', N)} \frac{[\mt(\varepsilon_M)]}{|\Hom(L', N)|} \right) = [K] \diamond ([M] \diamond [N]).$$
In the last but one equation, we use that $\left\langle K', M \right\rangle = \frac{|\Hom(L', M)|}{|\Hom(M', M)|},$ cf. the end of the proof of Proposition \ref{welldef}. We proved the first of identities (\ref{compat}); the proof of the second one is dual.
\end{proof}

Since the homomorphism $\diamond$ is compatible with the bimodule structure, it can be considered as a module homomorphism 
$$\mathcal{M}_1^{'}(\mathcal{E}) \times \mathcal{M}_1^{'}(\mathcal{E}) \to \mathcal{M}(\mathcal{E}).$$
Since it is compatible with relations (\ref{Zrelations}), it descends on $\mathcal{M}_2^{'}(\mathcal{E})$ and, moreover, defines a multiplication
$$\diamond: \mathcal{M}(\mathcal{E}) \times \mathcal{M}(\mathcal{E}) \to \mathcal{M}(\mathcal{E}).$$

\begin{definition}
We define the {\it semi-derived Hall algebra} $\mathcal{SDH(E)}$ as $\mathcal{M}(\mathcal{E})$ with the multiplication $\diamond.$
\end{definition}

\begin{remark}
It is easy to check that the product $\diamond$ given by formulae (\ref{diamonddef1})--(\ref{diamonddef2}) is the unique $\mathbb{T}_{ac}(\mathcal{E})-$bilinear multiplication on $\mathcal{M(E)}$ which coincides with that of the classical Hall algebra $\mathcal{H}(\mathcal{C}^b(\mathcal{E})$ on the pairs $(L', M)$ satisfying (\ref{extpmult}).
\end{remark}

We can give an alternative definition of the multiplication, where the summation is taken over the set of the isomorphism classes in the bounded derived category. It turns out that this gives us the structure constants.

\begin{definition}
Given a complex $M \in \mathcal{C}^b(\mathcal{E}),$ denote by $\overline{M}$ its isomorphism class in the bounded derived category $\mathcal{D}^b(\mathcal{E}).$ For any $\alpha \in \Iso(\mathcal{D}^b(\mathcal{E}))$ and
 any $L, M \in \mathcal{C}^b(\mathcal{E}),$ define the subset 
$$\Ext^1_{\mathcal{C}^b(\mathcal{E})}(A, B)_{\alpha} \subset \Ext^1_{\mathcal{C}^b(\mathcal{E})}(A, B)$$
as the set of all extensions of $L$ by $M$ whose middle term belongs to $\alpha.$ 
\end{definition}

\begin{lemma} \label{qisextclass}
Assume that for $E, E' \in \mathcal{C}^b(\mathcal{E}),$ we have:
\begin{itemize}
\item[1)] $\overline{E} = \overline{E'} = \alpha$ and
\item[2)] there exist $A, B \in \mathcal{C}^b(\mathcal{E}),$ such that 
$$\Ext^1_{\mathcal{C}^b(\mathcal{E})}(A, B)_E \neq 0, \Ext^1_{\mathcal{C}^b(\mathcal{E})}(A, B)_{E'} \neq 0.$$
\end{itemize}
Then in $\mathcal{M(E)}$ these two complexes determine the same element: $[E] = [E'].$
\end{lemma}

\begin{proof}
By condition 1), both $[E]$ and $[E']$ belong to the same component $\mathcal{M(E)}_{\alpha}$ of $\mathcal{M(E)}.$ By Theorem \ref{zfree}, there exists a unique element $t$ of the quantum torus $\mathbb{T}_{ac}(\mathcal{E}),$ such that $[E'] = t \diamond [E].$ By condition 2), the classes of $[E]$ and $[E']$ in $K_0(\mathcal{C}^b(\mathcal{E}))$ coincide, hence the class of $t$ in $K_0(\mathcal{C}^b(\mathcal{E}))$ equals zero. By Lemma \ref{k0inj}, this implies that the class of $t$ in $K_0(\mathcal{C}^b_{ac}(\mathcal{E}))$ equals zero. Therefore, we have $t = [0]$ in $\mathbb{T}_{ac}(\mathcal{E})).$ We finally obtain that $[E'] = [0] \diamond [E] = [E],$ q.e.d.
\end{proof}

\begin{corollary} \label{qisextclassident}
For all objects $A, B$ of $\mathcal{C}^b(\mathcal{E}),$ we have the following identity:
\begin{equation} \label{eeps}
\sum\limits_{E \in \Iso (\mathcal{C}^b(\mathcal{E}))} (|\Ext^1_{\mathcal{C}^b(\mathcal{E})}(A, B)_E| [E]) = \sum\limits_{\alpha \in \Iso (\mathcal{D}^b(\mathcal{E}))} (|\Ext^1_{\mathcal{C}^b(\mathcal{E})}(A, B)_{\alpha}| [E_{\alpha, A, B}]), 
\end{equation}
where for each $\alpha,$ the complex $\quad E_{\alpha, A, B}$ is the middle term of any extension belonging to $\Ext^1_{\mathcal{C}^b(\mathcal{E})}(A, B)_{\alpha}.$ 
\end{corollary}

\begin{proof}
Lemma \ref{qisextclass} implies that the right-hand side is well-defined. Now, for any $\alpha \in \Iso (\mathcal{D}^b(\mathcal{E})),$ we have, just by definition, 
$$\sum\limits_{E \in \Iso (\mathcal{C}^b(\mathcal{E})), \overline{E} = \alpha} |\Ext^1_{\mathcal{C}^b(\mathcal{E})}(A, B)_E|  = |\Ext^1_{\mathcal{C}^b(\mathcal{E})}(A, B)_{\alpha}|$$
and if this number is non-zero, then for each $E$ from the left-hand side, we have $[E] = [E_{\alpha, A, B}]$ in $\mathcal{M}(\mathcal{E}).$
\end{proof}

\begin{corollary} \label{qisextclassmult}
For each pair of bounded complexes $L, M \in \mathcal{C}^b(\mathcal{E}),$ the product $[L] \diamond [M]$ in the semi-derived Hall algebra is equal to the following sum:
\begin{equation}
[L] \diamond [M] = \frac{1}{\left\langle K, L \right\rangle} [K]^{-1} \diamond \sum\limits_{\alpha \in \Iso (\mathcal{D}^b(\mathcal{E}))} (\frac{|\Ext^1_{\mathcal{C}^b(\mathcal{E})}(L', M)_{\alpha}|}{|\Hom(L', M)|} [E_{\alpha, L', M}]),
\end{equation}
where 
$$K \rightarrowtail L' \stackrel{\qis}{\twoheadrightarrow} L $$ 
is a conflation such that $L'$ satisfies
$$\Ext^p_{\mathcal{C}^b(\mathcal{E})} (L', M) = \Ext^p_{\mathcal{D}^b(\mathcal{E})} (L', M), \quad \forall p > 0.$$ 
\end{corollary}

\section{Main properties}
\subsection{Associativity}

\begin{theorem} \label{zassoc}
For each triple of bounded complexes $A, B, C \in \mathcal{C}^b(\mathcal{E}),$ we have
$$([A] \diamond [B]) \diamond [C] = [A] \diamond ([B] \diamond [C]).$$
\end{theorem}

\begin{proof}
We have
$$([A] \diamond [B]) \diamond [C] = (\frac{1}{\left\langle K_1, A \right\rangle} [K_1]^{-1} \diamond \sum\limits_{\varepsilon \in \Ext^1_{\mathcal{C}^b(\mathcal{E})}(L_1,B)}(\frac{[\mt(\varepsilon)]}{|\Hom(L_1, B)|})) \diamond [C] =$$
\medskip
$$= \frac{1}{\left\langle K_1, A \right\rangle} [K_1]^{-1} \diamond \sum\limits_{\varepsilon \in \Ext^1_{\mathcal{C}^b(\mathcal{E})}(L_1,B)} (\frac{1}{|\Hom(L_1, B)|} ([\mt(\varepsilon)] \diamond [C])).$$
Using part (ii) of condition (C4), we obtain that for each extension $\varepsilon$ of $L_1$ by $B$ represented by a conflation $B \rightarrowtail X \twoheadrightarrow L_1,$ there exists a unique extension $\varepsilon'$ of $L_1$ by $L_2$ represented by $L_2 \rightarrowtail Y \twoheadrightarrow L_1,$ such that the following diagram commutes:

\begin{equation}  
\xymatrix@R=0.6cm{
*+++{K_2} \ar@{>->}[d] \ar@{=}[r] &*+++{K_2} \ar@{>->}[d] \ar@{->>}[r] & *+++{0} \ar@{>->}[d]\\
*+++{L_2} \ar@{->>}[d] \ar@{>->}[r]  & *+++{Y} \ar@{->>}[d] \ar@{->>}[r]  & *+++{L_1} \ar@{=}[d] \\
*+++{B}  \ar@{>->}[r]  & *+++{X} \ar@{->>}[r]  & *+++{L_1.} \\
}
\end{equation}
We have 
$$[X] = \frac{1}{\left\langle K_2, X \right\rangle} [K_2]^{-1} \diamond [Y] = \frac{1}{\left\langle K_2, B \right\rangle \left\langle K_2, L_1 \right\rangle} [K_2]^{-1} \diamond [Y],$$
i.e. 
$$[\mt(\varepsilon)] = \frac{1}{\left\langle K_2, B \right\rangle \left\langle K_2, L_1 \right\rangle} [K_2]^{-1} \diamond [\mt(\varepsilon')].$$
Then, by part (iii) of condition (C4), the expression above can be rewritten as follows:
$$\frac{1}{\left\langle K_1, A \right\rangle \left\langle K_2, B \right\rangle \left\langle K_2, L_1 \right\rangle} ([K_1]^{-1} \diamond [K_2]^{-1}) \diamond \frac{|\Hom(L_1, L_2)|}{|\Hom(L_1, B)|} \times$$
$$\times(\sum\limits_{\varepsilon' \in \Ext^1_{\mathcal{C}^b(\mathcal{E})}(L_1,L_2)} \frac{1}{|\Hom(L_1, L_2)|} [\mt(\varepsilon')] \diamond [C]) =$$
$$= \frac{|\Hom(L_1, L_2)|}{\left\langle K_1, A \right\rangle \left\langle K_2, B \right\rangle \left\langle K_2, L_1 \right\rangle |\Hom(L_1, B)|} ([K_1]^{-1} \diamond [K_2]^{-1}) \diamond (([L_1] \bullet [L_2]) \diamond [C]),$$
where $\bullet$ is the usual Hall product in the exact category $\mathcal{C}^b(\mathcal{E}).$ Thus, by part (iii) of  condition (C4), both multiplications in the formula $(([L_1] \diamond [L_2]) \diamond [C])$ are the usual Hall multiplications. Therefore, we can use the associativity property for this expression and rewrite the last formula in the above calculations as follows: 
$$([A] \diamond [B]) \diamond [C] = \frac{|\Hom(L_1, L_2)|}{\left\langle K_1, A \right\rangle \left\langle K_2, B \right\rangle \left\langle K_2, L_1 \right\rangle |\Hom(L_1, B)|} ([K_1]^{-1} \diamond [K_2]^{-1}) \diamond ([L_1] \bullet ([L_2] \bullet [C])).$$
Using the associativity of the product with $[K]^{-1}$ and the commutation rule from Lemma \ref{acinvcomm} for $[L_1]$ and $[K_2]^{-1},$ we obtain that this is equal to:
$$\frac{|\Hom(L_1, L_2)|}{\left\langle K_1, A \right\rangle \left\langle K_2, B \right\rangle \left\langle K_2, L_1 \right\rangle |\Hom(L_1, B)|} [K_1]^{-1} \diamond \frac{1}{|\Hom(L_1, K_2)|} \left\langle K_2, L_1 \right\rangle [L_1] \bullet ([K_2]^{-1} \diamond ([L_2] \bullet [C])).$$
Just by definition (and since the two $\left\langle K_2, L_1 \right\rangle$ cancel each other), it equals
 $$\frac{|\Hom(L_1, L_2)|}{\left\langle K_1, A  \right\rangle |\Hom(L_1, B)| |\Hom(L_1, K_2)|} [K_1]^{-1} \diamond [L_1] \bullet ([B] \diamond [C]) = $$
$$ = \frac{|\Hom(L_1, L_2)|}{|\Hom(L_1, B)| |\Hom(L_1, K_2)|} [A] \diamond ([B] \diamond [C]).$$
It is easy to check that part (ii) of condition (C4) implies that
$$\frac{|\Hom(L_1, L_2)|}{|\Hom(L_1, B)| |\Hom(L_1, K_2)|} = 1;$$
therefore, we have
$([A] \diamond [B]) \diamond [C] = [A] \diamond ([B] \diamond [C]).$
\end{proof}

\subsection{Derived invariance}

In this section, we prove that our algebra is invariant under some (quite big) class of derived equivalences:

\begin{proposition} \label{equivdbcb}
Suppose that $F: \mathcal{E'} \to \mathcal{E}$ is an exact functor between exact categories inducing an equivalence of bounded derived categories
$$F: \mathcal{D}^b(\mathcal{E'}) \overset\sim\to \mathcal{D}^b(\mathcal{E}).$$
Then the bounded derived categories of their categories of bounded complexes are equivalent via $F$:
$$F: \mathcal{D}^b(\mathcal{C}^b(\mathcal{E'})) \overset\sim\to \mathcal{D}^b(\mathcal{C}^b(\mathcal{E})).$$
\end{proposition}

\begin{proof}
Denote by $\mathcal{C}^{[-N,N]}(\mathcal{E})$ the category of complexes over $\mathcal{E},$ concentrated in degrees $- N, - N + 1,\ldots, N.$
Since we have
$$\mathcal{D}^b(\mathcal{C}^b(\mathcal{E'})) = \bigcup\limits_{N=0}^{\infty} \mathcal{D}^{b}(\mathcal{C}^{[-N,N]}(\mathcal{E'})),$$
and similarly for $\mathcal{E},$ 
it is enough to prove that $F$ induces an equivalence
$$\mathcal{D}^b(\mathcal{C}^{[-N, N]}(\mathcal{E'})) \overset{\overset{F}\sim}\to \mathcal{D}^b(\mathcal{C}^{[-N, N]}(\mathcal{E})),$$
for all $N \in \mathbb{Z}_{\geq 0}.$ For this we use the same ``d\'{e}vissage'' trick as in the proof of Proposition \ref{homfin}. 
We start by recalling the family of evaluation functors
$$ev_k: \mathcal{C}^{[-N,N]}(\mathcal{E}) \to \mathcal{E}, k = - N, - N + 1,\ldots, N $$
$$(\ldots\to 0 \to A^{-N} \to \ldots \to A^k \to \ldots \to A^N \to 0 \to \ldots) \quad \overset{ev_k}\mapsto \quad A^k.$$
Each of these functors is exact and admits an exact left adjoint
$ev_{k}^{\lambda}: \mathcal{E} \to \mathcal{C}^{[-N,N]}(\mathcal{E}):$
$$A \quad \overset{ev_{-N}^{\lambda}}\mapsto \quad (\ldots\to 0 \to A \to 0 \to \ldots \to 0 \to \ldots \to 0 \to 0 \to \ldots),$$
$$A \quad \overset{ev_{k}^{\lambda}}\mapsto \quad (\ldots\to 0 \to 0 \to 0 \to \ldots \to A = A \to \ldots \to 0 \to 0 \to \ldots),$$ 
for $- N + 1 \leq k \leq N.$
Here $A$ sits in the degree $- N$ and in degrees $k - 1$ and $k,$ respectively.
Since $ev_k$ and $ev_{k}^{\lambda}$ are exact adjoint, they induce a pair of adjoint functors between $\mathcal{D}^b(\mathcal{C}^{[-N, N]}(\mathcal{E}))$ and $\mathcal{D}^b(\mathcal{E}),$ which we will denote by the same symbols; this holds for each $k = - N, - N + 1,\ldots, N.$
We consider similar functors $ev_k'$ and ${ev_{k}'}^{\lambda}$ for $\mathcal{E'}.$
Then for each $X, Y \in \mathcal{D}^b(\mathcal{E'}), k, l \in [- N, N],$ we have isomorphisms
$$\Hom_{\mathcal{D}^b(\mathcal{C}^{[-N, N]}(\mathcal{E'}))} (ev_k^{\lambda} X, ev_{l}^{\lambda} Y) = \Hom_{\mathcal{D}^b(\mathcal{E'})} (X, ev_k \circ ev_{l}^{\lambda} Y) =$$
$$=  \Hom_{\mathcal{D}^b(\mathcal{E})} (FX, F (ev_k \circ ev_l^{\lambda} Y)) = \Hom_{\mathcal{D}^b(\mathcal{E})} (FX, ev_{k}' \circ {ev_l'}^{\lambda} (FY)) =$$
$$= \Hom_{\mathcal{D}^b(\mathcal{C}^{[-N, N]}(\mathcal{E'}))} ({ev_{k}'}^{\lambda} FX, {ev_{l}'}^{\lambda} FY).$$ 
It follows that it remains to show that the images of $\mathcal{D}^b(\mathcal{E})$ under $ev^{\lambda}_{k}, k = - N, - N + 1,\ldots, N$ generate $\mathcal{D}^b(\mathcal{C}^{[-N, N]}(\mathcal{E}))$ by triangles, and similarly for $\mathcal{E'}.$ Let us prove that these images generate the subcategories $\mathcal{D}^b(\mathcal{C}^{b}_k(\mathcal{E'}))$ and $\mathcal{D}^b(\mathcal{C}^{b}_k(\mathcal{E}))$ generated by the complexes concentrated in degree $k,$ for $k \in [- N, N].$ We give the proof by induction on $k,$ and write it only for $\mathcal{E},$ since for $\mathcal{E'}$ everything is the same. For $k = -N$ the statement is clear by the definition of $ev_{-N}^{\lambda}.$ Assume that we have proved the statement for $k \leq l.$ Consider an arbitrary stalk complex 
$$A^{\bullet} := \ldots \to 0 \to \ldots \to 0 \to A \to 0 \to \ldots \to 0 \to \ldots,$$
with $A$ sitting in degree $l.$ Then there is a conflation of complexes
$\Sigma^{-1} A^{\bullet}   \rightarrowtail  C(1_{\Sigma^{-1} A^{\bullet}})  \twoheadrightarrow A^{\bullet}$
inducing a triangle in $\mathcal{D}^b(\mathcal{C}^{[-N, N]}(\mathcal{E})).$ But $\Sigma^{-1} A^{\bullet} \in \mathcal{D}^b(\mathcal{C}^{b}_{l-1}(\mathcal{E})),$ hence it is in the subcategory of $\mathcal{D}^b(\mathcal{C}^{[-N, N]}(\mathcal{E}))$ generated by the images
$C(1_{\Sigma^{-1} A^{\bullet}}) = ev_{l}^{\lambda} (A).$ Therefore all stalk complexes in $\mathcal{D}^b(\mathcal{C}^{b}_l(\mathcal{E}))$ are in the subcategory generated by the images of $ev_{k}^{\lambda},$ but the stalk complexes generate by triangles the whole $\mathcal{D}^b(\mathcal{C}^{b}_l(\mathcal{E})),$ cf. the proof of Lemma \ref{devis}. 

Now, similarly to Lemma \ref{devis}, we verify easily that the closure under extensions of the $\mathcal{D}^b(\mathcal{C}^{b}_l(\mathcal{E})),$ for $l = - N, - N + 1,\ldots, N,$ is $\mathcal{D}^b(\mathcal{C}^{[-N, N]}(\mathcal{E})),$ and this completes the proof.
\end{proof}

\begin{proposition} \label{dbcbequiv}
Under the assumptions of Proposition \ref{equivdbcb}, we have an equivalence of derived categories of the categories of acyclic bounded complexes:
$$\mathcal{D}^b(\mathcal{C}_{ac}^b(\mathcal{E'})) \overset{\overset{F}\sim}\to \mathcal{D}^b(\mathcal{C}_{ac}^b(\mathcal{E}))$$
which induces an isomorphism of the Grothendieck groups
$${\bf K_0}(\mathcal{C}_{ac}^b(\mathcal{E'})) \overset\sim\to {\bf K_0}(\mathcal{C}_{ac}^b(\mathcal{E}))$$
and respects the Euler form $\left\langle \cdot, \cdot \right\rangle.$
\end{proposition}

\begin{proof}
By the equivalence of $\mathcal{D}^b(\mathcal{C}^b(\mathcal{E'}))$ and $\mathcal{D}^b(\mathcal{C}^b(\mathcal{E'})),$ it is sufficient to prove that the canonical functors
\begin{equation} \label{acff}
\mathcal{D}^b(\mathcal{C}_{ac}^b(\mathcal{E'})) \hookrightarrow \mathcal{D}^b(\mathcal{C}^b(\mathcal{E'})), \mathcal{D}^b(\mathcal{C}_{ac}^b(\mathcal{E})) \hookrightarrow \mathcal{D}^b(\mathcal{C}^b(\mathcal{E}))
\end{equation}
are fully faithful and that an object $A'$ of $\mathcal{D}^b(\mathcal{C}^b(\mathcal{E'}))$ lies in the essential image of $\mathcal{D}^b(\mathcal{C}^b_{ac}(\mathcal{E'}))$ iff $F(A')$ lies in the essential image of $\mathcal{D}^b(\mathcal{C}^b_{ac}(\mathcal{E})).$
The functors (\ref{acff}) are fully faithful by \cite[Theorem~12.1]{Kel2} applied to $\mathcal{B} = \mathcal{C}_{ac}^b(\mathcal{E}), \mathcal{A} = \mathcal{C}^b(\mathcal{E}).$ The followin lemma provides a functorial characterization of complexes of complexes, isomorphic in $\mathcal{D}^b(\mathcal{C}^b(\mathcal{E}))$ to complexes of acyclics, that gives us the desired condition on essential images and, therefore, completes the proof. 

\begin{lemma} \label{tot}
A complex $X \in \mathcal{C}^b(\mathcal{C}^b(\mathcal{E}))$ is isomorphic in $\mathcal{D}^b(\mathcal{C}^b(\mathcal{E}))$ to an object of $\mathcal{D}^b(\mathcal{C}_{ac}^b(\mathcal{E}))$ if and only if its total complex $\Tot(X)$ is acyclic, i.e. the functor
$$\Tot: \mathcal{D}^b(\mathcal{C}^b(\mathcal{E})) \to \mathcal{D}^b(\mathcal{E})$$
sends $X$ to 0.
\end{lemma}

\noindent{\it Proof.\/}
We start by an argument similar to Lemma \ref{projres}: for each complex of complexes $X \in \mathcal{C}^b(\mathcal{C}^b(\mathcal{E}))$ there exists a deflation quasi-isomorphism 
$X' \stackrel{\qis}{\twoheadrightarrow} X,$
where all but one of the components of $X'$ are acyclic. Indeed, for each complex $X^p \in \mathcal{C}^b(\mathcal{E})$ there exists a deflation from an acyclic  complex $C(1_{\Sigma^{-1} X})$ to $X^p.$ Without loss of generality, we can assume that $X$ is of form 
$$X: \ldots \to 0  \to X^0 \to \ldots \to X^n \to 0 \to \ldots, \quad X^p \in \mathcal{C}^b(\mathcal{E}).$$
We will construct the components of $X'$ step by step. Take a deflation $A^n \twoheadrightarrow X^n$ with $A^n$ acyclic. Then take the pullback of a morphism $X^{n-1} \to X^n$ and this deflation, and define $A^{n-1}$ (with a morphism $A^{n-1} \to A_n$ and a deflation $A^{n-1} \twoheadrightarrow X^{n-1})$ by an acyclic deflation to this pullback:
\begin{equation} \label{diag}
\xymatrix@R=0.6cm{
*+++{A^{n-1}} \ar@{.>>}[ddr] \ar@{->>}[dr] \ar@{.>}[drr] &&&\\
& *+++{\bullet} \ar@{->}[r] \ar@{->>}[d]  & *+++{A^n} \ar@{->>}[d] \ar@{->}[r] & *+++{0} \ar@{=}[d]\\
& *+++{X^{n-1}}  \ar@{->}[r]  & *+++{X^n} \ar@{->}[r]  & *+++{0}. \\
}
\end{equation}
In similar manner we construct $A^{n-2}, \ldots, A^1, A^0.$ At the last step, we do not consider a deflation from an acyclic complex, just a pullback $B$ given by the diagram
\begin{equation}
\xymatrix@R=0.6cm{
*+++{B} \ar@{->}[r] \ar@{->>}[d]  & *+++{A^0} \ar@{->>}[d]\\
*+++{0}  \ar@{->}[r]  & *+++{X^0}.\\
}
\end{equation}
By construction, we get a deflation quasi-isomorphism $X' \twoheadrightarrow X,$ where
$$X' = \ldots \to 0 \to B \to A^0 \to \ldots \to A^n \to 0 \to \ldots .$$

Now we see that $X$ is isomorphic in $\mathcal{D}^b(\mathcal{C}^b(\mathcal{E}))$ to a complex from $\mathcal{D}^b(\mathcal{C}^b_{ac}(\mathcal{E}))$ if and only if $X'$ is. But all $A^i$ are acyclic, hence this last condition holds if and if $B$ is acyclic. Clearly, in this and only this case we have $\Tot(X') = 0.$ But $\Tot$ is a functor from the derived category $\mathcal{D}^b(\mathcal{C}^b(\mathcal{E})),$ which gives us the desired criterion.
\end{proof}

\begin{corollary} \label{toriequiv}
Under the assumptions of Proposition \ref{equivdbcb}, we have an isomorphism of quantum tori of acyclic complexes:

$$\mathbb{T}_{ac}(\mathcal{E'}) \overset\sim\to \mathbb{T}_{ac}(\mathcal{E}).$$
\end{corollary}

\begin{theorem} \label{sdhequiv}
Under the assumptions of Proposition \ref{equivdbcb}, we have an isomorphism of algebras:
$$\mathcal{SDH}(\mathcal{E'}) \overset\sim\to \mathcal{SDH}(\mathcal{E}).$$
\end{theorem}

\begin{proof}
By the Corollary \ref{toriequiv}, we have an isomorphism
$\mathcal{M}(\mathcal{E'}) \overset\sim\to \mathcal{M}(\mathcal{E})$
of free modules over isomorphic quantum tori with bases which are in bijection by the derived equivalence. Since we can consider all bounded complexes as stalk complexes in $\mathcal{D}^b(\mathcal{C}^b(\mathcal{E'}))$ and $\mathcal{D}^b(\mathcal{C}^b(\mathcal{E}))$ respectively, all extensions in the categories of bounded complexes are just morphisms in these derived categories, hence they are preserved under the functor $\mathcal{D}^b(\mathcal{C}^b(\mathcal{E'})) \to \mathcal{D}^b(\mathcal{C}^b(\mathcal{E}))$ induced by $F.$ It follows that the multiplication of the  semi-derived Hall algebra is preserved as well, which completes the proof.
\end{proof}

\section{Case with enough projectives}

Throughout this section, we assume that our category $\mathcal{E}$ satisfies the following conditions: 
\begin{itemize}
\item[1)] $\mathcal{E}$ has enough projectives;
\item[2)] each object in $\mathcal{E}$ has a finite projective resolution.
\end{itemize}

Here we record some simple facts concerning extensions in the category $\mathcal{C}^b(\mathcal{E}).$

\begin{lemma} \label{acproj}
For each complex $K \in \mathcal{C}^b_{ac}(\mathcal{E}),$ there exists an acyclic complex of complexes
\begin{equation} \label{projacres}
0 \to P_d \to P_{d-1} \to \ldots \to P_1 \to P_0 \to K \to 0,
\end{equation}
where $d$ is the maximum of the projective dimensions of the components of $K;$ and the $P_i$ are bounded acyclic complexes of projectives.
\end{lemma}

\begin{proof}
By Lemma \ref{projres}, there exists a conflation whose deflation is a quasi-isomorphism
$$K_0 \rightarrowtail P_0 \stackrel{\qis}{\twoheadrightarrow} K,$$ 
with $P_0 \in \mathcal{C}^b(\mathcal{P}).$ Then $P_0$ is acyclic, and from the long exact sequences of extensions of components by arbitrary objects of $\mathcal{E},$ it follows that the components of $K_1$ are of projective dimension at most $d - 1.$ Now we prove the statement by induction on the maximal projective dimension of the components of our acyclic complex.
\end{proof}

\begin{corollary} \label{arbproj}
For each complex $M \in \mathcal{C}^b(\mathcal{E}),$ there exists an acyclic complex of complexes
\begin{equation} \label{arbprojeq}
0 \to P_d \to P_{d-1} \to \ldots \to P_1 \to P_0 \stackrel{\qis}{\twoheadrightarrow} M \to 0,
\end{equation}
where $d$ is the maximum of the projective dimensions of the components of $M$ and the $P_i$ are bounded complexes of projectives which are acyclic for $i \geq 1.$
\end{corollary}

\begin{proof}
By Lemma \ref{projres} there exists a conflation whose deflation is a quasi-isomorphism
$$K_0 \rightarrowtail P_0 \stackrel{\qis}{\twoheadrightarrow} M,$$
with $P_0 \in \mathcal{C}^b(\mathcal{P}).$ Now apply Lemma \ref{acproj} to $K_0,$ whose components have projective dimension at most $d - 1.$
\end{proof}

\begin{corollary} \label{achomfin}
For each acyclic complex $K \in \mathcal{C}_{ac}^b(\mathcal{E})$ and each $L \in \mathcal{C}^b(\mathcal{E})$ the following holds:
$$\Ext^p(K, L) = 0, \quad \forall p > d,$$
where $d$ is the maximum of projective dimensions of the components of $K.$
\end{corollary}

\begin{proof}
By Proposition \ref{projext}, the complex (\ref{projacres}) is a projective resolution of $K$ in $\mathcal{C}^b(\mathcal{E}).$
\end{proof}

\begin{corollary} \label{acext}
For each pair of complexes $L, M \in \mathcal{C}^b(\mathcal{E})$ the following holds:
$$\Ext^p_{\mathcal{C}^b(\mathcal{E})} (M, L) = \Ext^p_{\mathcal{D}^b(\mathcal{E})} (M, L), \quad \forall p > d,$$
where $d$ is the maximum of projective dimensions of the components of $M.$
\end{corollary}

\begin{proof}
Write down the long exact sequences of extensions in $\mathcal{C}^b(\mathcal{E})$ and in $\mathcal{D}^b(\mathcal{E})$ induced by morphisms from elements of a conflation
$K_0 \rightarrowtail P_0 \stackrel{\qis}{\twoheadrightarrow} M,$ 
with $P_0 \in \mathcal{C}^b(\mathcal{P}),$ to $L;$ and apply Corollary \ref{acext} to $K_0$ and $L$ and the five-lemma to these two sequences. Note that the maximum of projective dimensions of $K_0$ is $d - 1.$
\end{proof}

With the conditions 1) and 2) on $\mathcal{E},$ Proposition \ref{projres} establishes that the inclusion
$\mathcal{P} \hookrightarrow \mathcal{E}$
induces an equivalence of derived categories
$\mathcal{D}^b(\mathcal{P}) \overset\sim\to \mathcal{D}^b(\mathcal{E}).$
Then Theorem \ref{sdhequiv} provides an isomorphism
\begin{equation} \label{allproj}
I: \mathcal{SDH}(\mathcal{P}) \overset\sim\to \mathcal{SDH}(\mathcal{E}), \quad [P] \mapsto [P].
\end{equation}

\begin{proposition} \label{sdhpequiv}
The map 
$$F: [M] \mapsto \frac{1}{|\Hom(\bigoplus\limits_{k \in \mathbb{Z}_{\geq 0}} P_{2k+1}, M)|} [\bigoplus\limits_{k \in \mathbb{Z}_{\geq 0}} P_{2k+1}]^{-1} \diamond [\bigoplus\limits_{k \in \mathbb{Z}_{\geq 0}} P_{2k}];$$ 
$$[K]^{-1} \mapsto [F([K])]^{-1}, \quad K \in \mathcal{C}_{ac}^b(\mathcal{E})$$
provides an isomorphism 
$$\mathcal{SDH}(\mathcal{E}) \overset\sim\to \mathcal{SDH}(\mathcal{P}),$$
inverse to the isomorphism (\ref{allproj}). 
\end{proposition}

\noindent It is easy to check that $F$ is inverse to $I$ as a map. $I$ being an isomorphism, so is $F.$

Note that all extensions in $\mathcal{C}^b(\mathcal{P})$ are isomorphic to those in $\mathcal{D}^b(\mathcal{P}),$ therefore the multiplication in $\mathcal{SDH}(\mathcal{P})$ is just the usual Hall multiplication in the exact category $\mathcal{C}^b(\mathcal{P}).$ It means that $\mathcal{SDH}(\mathcal{P})$ is the usual Hall algebra of $\mathcal{C}^b(\mathcal{P}),$ localized at the classes of acyclic complexes (all of which are contractible):

\begin{proposition} \label{sdhp}
$\qquad\quad \mathcal{SDH}(\mathcal{P}) = \mathcal{H}(\mathcal{C}^b(\mathcal{P}))[[K]^{-1}| K \in \mathcal{C}^b(\mathcal{P})].$
\newline If $\mathcal{E}$ has enough projectives and any object in $\mathcal{E}$ has finite projective resolution, then
$$\mathcal{SDH}(\mathcal{E}) = \mathcal{H}(\mathcal{C}^b(\mathcal{P}))[[K]^{-1}| K \in \mathcal{C}^b(\mathcal{P})].$$
\end{proposition}

\subsection{Hereditary case}

It turns out that if the category $\mathcal{E}$ is hereditary, i.e of global dimension 1, and has enough projectives, then the multiplication in our algebra is given by the same formula as the usual Ringel-Hall multiplication.

\begin{theorem} \label{hermult}
Assume that $\mathcal{E}$ is hereditary and has enough projectives. Then for $L, M \in \mathcal{C}^b(\mathcal{E}),$ the product $[L] \diamond [M]$ in the semi-derived Hall algebra is equal to the following sum:
$$[L] \diamond [M] = \sum\limits_{X \in \Iso(\mathcal{C}^b(\mathcal{E}))} \frac{|\Ext^1_{\mathcal{C}^b(\mathcal{E})} (L, M)_X|}{|\Hom(L, M)|} [X],$$
\end{theorem}

\begin{proof}
We can choose a conflation 
$K \rightarrowtail L' \stackrel{\qis}{\twoheadrightarrow} L$ 
with $K \in \mathcal{C}_{ac}^b(\mathcal{P}), L' \in \mathcal{C}^b(\mathcal{P})$ to write down the product $[L] \diamond [M]$ (if $L'$ has projective components, then, by the heredity assumption, $K$ has projective components as well). By Proposition \ref{projext}, we have
$$\Ext^p_{\mathcal{C}^b(\mathcal{E})} (K, M) = \Ext^p_{\mathcal{D}^b(\mathcal{E})} (K, M) = 0, \quad \forall p > 0,$$
hence
$\left\langle K, M \right\rangle = |\Hom(K, M)|.$
Consider the part of the long exact sequence of extensions in the category $\mathcal{C}^b(\mathcal{E}):$

\begin{equation} \label{herles}
0 \to \Hom(L, M) \to \Hom(L', M) \to \Hom(K, M) \to \atop \to \Ext^1_{\mathcal{C}^b(\mathcal{E})}(L, M) \to \Ext^1_{\mathcal{C}^b(\mathcal{E})}(L', M) \to \Ext^1_{\mathcal{C}^b(\mathcal{E})}(K, M) = 0.
\end{equation}
We find out that $\Ext^1_{\mathcal{C}^b(\mathcal{E})}(L, M)$ surjects onto $\Ext^1_{\mathcal{C}^b(\mathcal{E})}(L', M) = \Ext^1_{\mathcal{D}^b(\mathcal{E})}(L, M).$ That means that for each extension $\varepsilon \in \Ext^1_{\mathcal{C}^b(\mathcal{E})}(L, M)$ represented by 
$$M \rightarrowtail X \twoheadrightarrow L$$
there exists a unique extension $\varepsilon' \in \Ext^1_{\mathcal{C}^b(\mathcal{E})}(L', M)$ represented by
$$M \rightarrowtail E \twoheadrightarrow L'$$
such that the following diagram commutes:
\begin{equation}  
\xymatrix@R=0.6cm{
*+++{0} \ar@{>->}[d] \ar@{=}[r] &*+++{K} \ar@{>->}[d] \ar@{=}[r] & *+++{K} \ar@{>->}[d]\\
*+++{M} \ar@{=}[d] \ar@{>->}[r]  & *+++{E} \ar@{->>}[d] \ar@{->>}[r]  & *+++{L'} \ar@{->>}[d] \\
*+++{M}  \ar@{>->}[r]  & *+++{X} \ar@{->>}[r]  & *+++{L}, \\
}
\end{equation}
and for each $\varepsilon'$ there are exactly $\frac{|\Ext^1_{\mathcal{C}^b(\mathcal{E})}(L, M)|}{|\Ext^1_{\mathcal{C}^b(\mathcal{E})}(L', M)|}$ of $\varepsilon$ corresponding to $\varepsilon'.$ We see that for each such pair $(\varepsilon, \varepsilon'),$ we have
$$[\mt(\varepsilon')] = \left\langle K, M \right\rangle \left\langle K, L \right\rangle [K] \diamond [\mt(\varepsilon)],$$
cf. proofs of Proposition \ref{welldef} and Theorem \ref{zassoc}. We get
$$[L] \diamond [M] = \frac{1}{\left\langle K, L \right\rangle} [K]^{-1} \diamond \sum\limits_{\varepsilon' \in \Ext^1_{\mathcal{C}^b(\mathcal{E})}(L',M)}  \frac{[\mt(\varepsilon')]}{|\Hom(L', M)|} =$$
$$= \frac{1}{\left\langle K, L \right\rangle} [K]^{-1} \frac{|\Ext^1_{\mathcal{C}^b(\mathcal{E})}(L', M)|}{|\Ext^1_{\mathcal{C}^b(\mathcal{E})}(L, M)|} \diamond \sum\limits_{\varepsilon \in \Ext^1_{\mathcal{C}^b(\mathcal{E})}(L, M)} \frac{\left\langle K, M \right\rangle \left\langle K, L \right\rangle [K] \diamond [\mt(\varepsilon)]}{|\Hom(L', M)|} =$$
$$= (|\Hom(K, M)| \frac{|\Ext^1_{\mathcal{C}^b(\mathcal{E})}(L', M)|}{|\Ext^1_{\mathcal{C}^b(\mathcal{E})}(L, M)|} \frac{|\Hom(L, M)|}{|\Hom(L', M)|}) \sum\limits_{\varepsilon \in \Ext^1_{\mathcal{C}^b(\mathcal{E})}(L, M)} \frac{[\mt(\varepsilon)]}{|\Hom(L, M)|} =$$
$$= (|\Hom(K, M)| \frac{|\Ext^1_{\mathcal{C}^b(\mathcal{E})}(L', M)|}{|\Ext^1_{\mathcal{C}^b(\mathcal{E})}(L, M)|} \frac{|\Hom(L, M)|}{|\Hom(L', M)|}) \sum\limits_{X \in \Iso(\mathcal{C}^b(\mathcal{E}))} \frac{|\Ext^1_{\mathcal{C}^b(\mathcal{E})}(L, M)_X|}{|\Hom(L, M)|} [X].$$

It follows from the long exact sequence (\ref{herles}) that
$$|\Hom(K, M)| \frac{|\Ext^1_{\mathcal{C}^b(\mathcal{E})}(L', M)|}{|\Ext^1_{\mathcal{C}^b(\mathcal{E})}(L, M)|} \frac{|\Hom(L, M)|}{|\Hom(L', M)|} = 1,$$
and we find out the desired formula.
\end{proof}

\begin{corollary} \label{herquotloc}
Assume that $\mathcal{E}$ is hereditary and has enough projectives. Then there exists an algebra homomorphism 
$$p: \mathcal{H}(\mathcal{C}^b(\mathcal{E})) \to \mathcal{SDH(E)}.$$
The homomorphism $p$ induces an algebra isomorphism
\begin{equation}
(\mathcal{H}(\mathcal{C}^b(\mathcal{E}))/I)[S^{-1}] \overset\sim\to \mathcal{SDH(E)},
\end{equation}
where $I$ is the two-sided ideal generated by all differences
$[L] - [K\oplus M]$, where $K \rightarrowtail L \twoheadrightarrow M$ is a conflation in $\mathcal{C}^b(\mathcal{E})$ with acyclic $K,$ and $S$ is the set of all classes $[K]$ of acyclic complexes.
\end{corollary}

The following known result is useful for us to give a set of generators of $\mathcal{SDH(A)}:$

\begin{lemma} \label{herderdirstalk}
Suppose that $\mathcal{A}$ is a hereditary abelian category. Then each object in $\mathcal{C}^b(\mathcal{A})$ is
quasi-isomorphic to the direct sum of its homology objects, where $H^i$ is concentrated in the degree $i.$ Equivalently, each object in $\mathcal{D}^b(\mathcal{A})$ is isomorphic to a direct sum of indecomposable stalk complexes.
\end{lemma}

The proof can be found, e.g., in \cite{Kap}. Let us denote by $u_{A, m}$ a stalk complex with non-zero component $A,$ concentrated in the degree $m.$ Let us also define the complexes $v_{A, m}, A \in \mathcal{A},$ by
$$v_{A, m} := \ldots \to 0 \to 0 \to A  \xrightarrow[]{1_A} A \to 0 \to 0 \to \ldots,$$
where $A$ sits in degrees $m$ and $(m + 1).$  For any class $\alpha \in K_0(\mathcal{A})$ and for any integer $m,$ let us define the class $v_{\alpha, m}$ as the product 
$$[v_{A, m}] \diamond [v_{B, m}]^{-1},$$
for any $A, B \in \mathcal{A}$ such that $\alpha = A - B$ in the Grothendieck group $K_0(\mathcal{A}).$

From now on up to the end of the section, we suppose that $\mathcal{A}$ is a hereditary abelian category with enough projectives, satisfying all the conditions (C1)-(C3).  

\begin{proposition} \label{genset}
$\mathcal{SDH(A)}$ is generated by the classes of complexes $u_{A, m}$ and the classes $v_{\alpha, m}.$
\end{proposition}

\begin{proof}
From Theorem \ref{zfree}, it follows that $\mathcal{SDH(A)}$ has a basis over $\mathbb{Q},$ each of whose elements can be presented as the product of an element of the quantum torus by the class of an object in $\mathcal{D}^b(\mathcal{A}).$ It is easy to verify that for any acyclic bounded complex $M$ and any integer $i,$ we have $$[M] = a_i [\tau_{\geq i} M] \diamond [\tau_{< i} M],$$
for a certain $a_i \in \mathbb{Q}.$ Recall that $\tau_{\geq i} M, \tau_{< i} M$ are the intelligent truncations of $M.$ By induction, we prove that $\mathbb{A}_{ac}(\mathcal{A})$ is generated by the classes of the complexes $v_{A,m}.$ Then it is clear that $\mathbb{T}_{ac}(\mathcal{A})$ is generated by the  $v_{\alpha, m}.$ Similarly, we prove that the class of any direct sum  $\bigoplus\limits_{n_1 \leq i \leq n_2} A_i,$ where $A_i$ is a stalk complex concentrated in the degree $i,$ equals, up to a rational multiplicative coefficient, the product
$$[A_{n_1}] \diamond [A_{n_1 + 1}] \diamond \cdots \diamond [A_{n_2}],$$
i.e. a product of classes of $u_{A_i, i}.$
By Lemma \ref{herderdirstalk}, any quasi-isomorphism class in $\mathcal{C}^b(\mathcal{A})$ contains such a direct sum.
\end{proof}

\begin{lemma} \label{i_m}
There is a family of injective ring homomorphisms
$$I_m : \mathcal{H}(\mathcal{A}) \hookrightarrow \mathcal{SDH(A)}, \quad [A] \mapsto [u_{A,m}].$$
\end{lemma}

\begin{proof}
It is enough to prove that, for any $A, B \in \mathcal{A},$ we have
$$\sum\limits_{\mathcal{C}} \frac{|\Ext^1_{\mathcal{A}}(A, B)_C|}{|\Hom_{\mathcal{A}}(A, B)|} u_{C, m} = \sum\limits_{\mathcal{X}} \frac{|\Ext^1_{\mathcal{C}^b(\mathcal{A})}(u_{A, m}, u_{B, m})_{X}|}{|\Hom_{\mathcal{C}^b(\mathcal{A})}(u_{A, m}, u_{B, m})|} X,$$
which is easy to check.
\end{proof}

Let us write down some simple relations on generators in $\mathcal{SDH(A)}.$ 

\begin{lemma} \label{u_a,u_b}
Assume that for $A, B \in \mathcal{A},$ we have 
$\Hom(A, B) = 0.$
Then in $\mathcal{SDH(A)}$ we have
$$[u_{A,m}] \diamond [u_{B,n}] = [u_{A,m} \oplus u_{B,n}],$$
for any $m \neq n \in \mathbb{Z}.$
\end{lemma}

\begin{lemma} \label{u_a,u_a}
Assume that for $A \in \mathcal{A},$ we have 
$\dim \Hom(A, A) = 1.$
Assume also that $\mathcal{A}$ is linear over a finite field $k = \mathbb{F}_q.$ Then in $\mathcal{SDH(A)}$ we have
$$[u_{A,m}] \diamond [u_{A,n}] = [u_{A,m} \oplus u_{A,n}] + (q - 1) \delta_m^{n-1} [v_{A,n}],$$
for any $m \neq n \in \mathbb{Z}.$ As a corollary, we have
$$[v_{A,n}] = (q - 1)^{-1} ([u_{A,n}] \diamond [u_{A,(n+1)}] - [u_{A,(n+1)}] \diamond [u_{A,n}]), \quad \forall n \in \mathbb{Z}.$$
\end{lemma}

To prove both of these lemmas, one should use the form of the product given in Theorem \ref{hermult}. Note that 
$$\Hom_{\mathcal{C}^b(\mathcal{A})}(u_{A,m}, u_{B,n}) = 0, \quad \forall A, B \in \mathcal{A}, \forall m \neq n \in \mathbb{Z},$$
and $\Ext^1_{\mathcal{C}^b(\mathcal{A})}(u_{A,m}, u_{B,n})$ can be easily computed. 



We can also calculate the value of the Euler form on pairs of generators.

\begin{lemma} \label{eulerform}
For any $A, B \in \mathcal{A}$ and any $m, n \in \mathbb{Z},$ we have 
$$\left\langle v_{\alpha,m}, u_{B,n} \right\rangle = \delta_m^n \left\langle \alpha, B \right\rangle_{\mathcal{A}}; \quad \left\langle v_{\alpha,m}, v_{\beta,n} \right\rangle = (\delta_m^n + \delta_m^{n+1}) \left\langle \alpha, \beta \right\rangle_{\mathcal{A}};$$
$$\left\langle u_{B,n}, v_{\alpha,m} \right\rangle = \delta_m^{n-1} \left\langle B, \alpha \right\rangle_{\mathcal{A}}.$$ 
\end{lemma}

\begin{corollary} \label{eulerform2}
For any $A, B \in \mathcal{A}$ and any $n > m \in \mathbb{Z},$ we have 
$$\left\langle u_{A,m}, u_{B,n} \right\rangle = (\left\langle A, B \right\rangle_{\mathcal{A}})^{(-1)^{(n-m)}}; \quad \left\langle u_{A,m}, u_{B,m} \right\rangle = \left\langle A, B \right\rangle_{\mathcal{A}};$$
$$\left\langle u_{B,n}, u_{A,m} \right\rangle = 1.$$
\end{corollary}

One may twist the multiplication in $\mathcal{SDH(A)}$ in different ways. Since the classes of the stalk complexes yield the full set of generators, we can define the twist of their products and then extend it by bilinearity. Here are some natural twists:

\begin{itemize}
\item[1)] $[u_{A,m}] *_1 [u_{B,n}] := \left\langle u_{A,m}, u_{B,n} \right\rangle [u_{A,m}] \diamond [u_{B,n}].$ With this twist, the quantum torus becomes commutative;
\item[2)] $[u_{A,m}] *_2 [u_{B,n}] := \sqrt{\left\langle u_{A,m}, u_{B,n} \right\rangle} [u_{A,m}] \diamond [u_{B,n}].$ This twist is probably the most natural analogue of Ringel's twist;
\item[3)] $[u_{A,m}] *_3 [u_{B,n}] := (\sqrt{\left\langle A, B \right\rangle_{\mathcal{A}}})^{\delta_m^n} [u_{A,m}] \diamond [u_{B,n}].$ This twist is a $\mathbb{Z}-$graded analogue of the twist used in \cite{B};
\item[4)] $[u_{A,m}] *_4 [u_{B,n}] := (\sqrt{\left\langle A, B \right\rangle_{\mathcal{A}}})^{(-1)^{(n-m)}} [u_{A,m}] \diamond [u_{B,n}].$ Note that here we take arbitrary $n, m \in \mathbb{Z},$ which need not satisfy $n > m.$ This twist is similar to the one used in \cite{HL}, \cite{Sh}.
\end{itemize}

In any of cases 2), 3) and 4), we get a family of embeddings of $\mathcal{H}_{tw}(\mathcal{A})$ into $\mathcal{SDH}_{tw}(\mathcal{A}),$ with the twist thus defined. It is easy to show that all of these twisted algebras are invariant under the same class of derived equivalences as in Theorem \ref{sdhequiv}, cf. Theorem \ref{z2sdhequivtw} below. 

\subsection{Example: quiver representations}

Let $Q$ be a simply-laced Dynkin quiver on the vertices ${1,\ldots, n}.$ Let $\mathcal{A}$ be the category of finite-dimensional representations of $Q$ over the field $k = \mathbb{F}_q.$ This abelian category satisfies all the assumptions (C1)-(C3); moreover, it is hereditary and has enough projectives. For each vertex $i$ of $Q,$ we denote the corresponding one-dimensional simple module by $S_i \in \mathcal{A}.$ Two objects associated to a quiver $Q,$ the Hall algebra $\mathcal{H}(\mathcal{A})$ and the quantum group $U_{\sqrt{q}}(\mathfrak{g})$ are related to each other by the following important result.

\begin{theorem} [Ringel, {\cite[Sch]{R1}}] \label{rgthm}
There are isomorphisms of algebras
$$R: U_{\sqrt{q}}(\mathfrak{n}^+) \overset\sim\to \mathcal{H}_{tw}(\rep_{\mathbb{F}_q} (Q)), \quad R^e : U_{\sqrt{q}}(\mathfrak{b}^+) \overset\sim\to H^e_{tw}(\rep_{\mathbb{F}_q} (Q)),$$
defined on generators by
$$R(E_i) = R^e(E_i) = \frac{[S_i]}{(q - 1)}, \quad R^e(K_i) = K_{S_i}.$$
\end{theorem}


In the algebra $\mathcal{SDH(A)},$ we choose a set of generators consisting of elements of two types, according to Proposition \ref{genset}: generators of the quantum torus $\mathbb{T}_{ac}(\mathcal{E})$ and classes of stalk complexes. Let $u_{i, m} := u_{S_i, m},$ for $i = 1,\ldots,n$ and $m \in \mathbb{Z}.$ 

\begin{proposition} The algebra $\mathcal{SDH(A)}$ is generated by the classes $[u_{i,m}], [v_{\alpha,m}],$ subject only to the following relations: for every $p > m + 1 \in \mathbb{Z}$ and all $1 \leq i, j \leq n$ and $\alpha, \beta \in K_0(\mathcal{A}),$ we have
\begin{itemize}
\item[(U)] the relations in $I_m(\mathcal{H(A)})$ for the $u_{i,m},$ $i = 1,\ldots,n;$
$$u_{i,m} u_{j,m+1} = u_{j,m+1} u_{i,m} + \delta_{i}^{j} (q - 1) v_{S_i,m}; \qquad\qquad 
u_{i,m} u_{j,p} = u_{j,p} u_{i,m};$$
\item[(V)] 
$$v_{\alpha, m} v_{\beta, m} = \frac{\langle \beta, \alpha \rangle}{\langle \alpha, \beta \rangle} v_{\beta, m} v_{\alpha, m}; \quad\quad 
v_{\alpha,m} v_{\beta,m+1} = \langle \beta, \alpha \rangle v_{\beta,m+1} v_{\alpha,m}; \quad\quad
v_{\alpha,m} v_{\beta,p} =  v_{\beta,p} v_{\alpha,m};$$
\item[(UV)] 
$$v_{\alpha,m} u_{j,m+1} = \langle S_j, \alpha \rangle u_{j,m+1} v_{\alpha,m}; \quad\quad
u_{i,m} v_{\beta,m} = \langle \beta, S_i \rangle v_{\beta,m} u_{i,m};$$
$$u_{i,m} v_{\beta,m+1} =  v_{\beta,m+1} u_{i,m}; \quad\quad
u_{i,m} v_{\beta,p} =  v_{\beta,p} u_{i,m}; \quad\quad
v_{\alpha,m} u_{j,p} =  u_{j,p} v_{\alpha,m}.$$ 
\end{itemize}
\end{proposition}

\begin{proof}
Let $B$ be the algebra generated by the classes $[u_{i,m}], [v_{\alpha,m}],$ subject only to the above relations. We have a natural map 
$$\phi: B \to \mathcal{SDH(A)}, [X] \mapsto [X].$$ 
Due to Theorem \ref{rgthm}, we know that the classes $[u_{i,m}]$ give a basis of $I_m(\mathcal{H(A)}),$ for any integer $m.$ Therefore, by Proposition \ref{genset}, the classes $[u_{i,m}], [v_{\alpha,m}]$ generate $\mathcal{SDH(A)}.$ Applying lemmas \ref{i_m}-\ref{eulerform} and Theorem \ref{rgthm}, one easily proves that all the relations in the statement are satisfied in the algebra $\mathcal{SDH(A)},$ i.e. the map $\phi$ is a surjective homomorphism. Let us prove its injectivity. Being the twisted group algebra of $K_0(\mathcal{C}^b(\mathcal{A})),$ the quantum torus has a basis consisting of the elements of this Grothendieck group. We know (cf. the proof of Lemma \ref{k0inj}) that $K_0 (\mathcal{C}^b_{ac}(\mathcal{E})) \cong \coprod\limits_{\mathbb{Z}} K_0 (\mathcal{E}).$ From here it immediately follows that the elements of the form 
$$[v_{\alpha_1,m_1}] \diamond [v_{\alpha_2,m_2}] \diamond\ldots\diamond [v_{\alpha_k,m_k}], m_1 > m_2 >\ldots> m_k,$$
form a basis of $\mathbb{T}_{ac}(\mathcal{A})$ over $\mathbb{F}_q.$ The set formed by the elements
$$([u_{i_1,l_1}] \diamond [v_{i_2,l_2}] \diamond\ldots\diamond [v_{i_k,l_k}]) \diamond ([v_{\alpha_1,m_1}] \diamond [v_{\alpha_2,m_2}] \diamond\ldots\diamond [v_{\alpha_k,m_k}]),$$
where $l_1 > l_2 >\ldots> l_k, m_1 > m_2 >\ldots> m_k,$
spans $B.$ Thanks to the above argument and the freeness of the algebra $\mathcal{SDH(A)}$ over the quantum torus, this set is mapped to a basis of $\mathcal{SDH(A)}.$ Thus, the homomorphism $\phi$ is injective, i.e. it is an isomorphism.
\end{proof}

Note that with any of the twists 2), 3) and 4), the relations (U1) will be transformed into the quantum Serre relations.

\section{Tilting objects and derived equivalences}

We say that $T$ is a {\it tilting} object in an exact category $\mathcal{E}$ satisfying conditions (C1) - (C4), if the following conditions hold:

\begin{itemize}
\item[(T1)] the groups $\Ext^i(T,T)$ vanish for $i>0;$
\item[(T2)] the full exact category $\add T$ containing all finite direct sums of indecomposable summands of $T$ generates $\mathcal{D}^b(\mathcal{A});$
\item[(T3)] the algebra $A = \End(T)$ is of finite global dimension.
\end{itemize}

Note that the exact structure on $\add T$ splits. Therefore, the inclusion
$I_T: \add(T) \to \mathcal{E}$
induces an equivalence 
$$\mathcal{D}^b(\add(T)) \overset\sim\to \mathcal{D}^b(\mathcal{E})$$
by conditions (T1) and (T2). 
Moreover, we have an exact fully faithful functor
$$\Hom(T,?): \add(T) \to \mod(A).$$
This functor induces an equivalence from $\add(T)$
onto the subcategory $\proj(A)$ of finitely generated projective
$A$-modules. Therefore, by condition (T3), $\Hom(T,?)$ also induces an equivalence
$$\mathcal{D}^b(\add(T)) \overset\sim\to \mathcal{D}^b(\mod (A)).$$




In particular, any complex $M \in \mathcal{C}^b(\mathcal{A})$ admits a finite $\mathcal{C}^b(\add T)-${\it coresolution}, i.e. an acyclic complex of complexes
\begin{equation} \label{factcores}
0 \to M \overset{\qis}\to T_0 \to T_1 \ldots \to T_d \to 0,
\end{equation}
where all $T_i$ belong to $\mathcal{C}^b(\add T),$ and all of them, except for $T_0,$ are acyclic. By Theorem \ref{sdhequiv} and similarly to Proposition \ref{sdhpequiv}, we obtain the following theorem.

\begin{theorem} \label{tiltequiv}
We have the following isomorphisms between the semi-derived Hall algebras:
$$\mathcal{SDH}(\mod(\End(T))) \overset\sim\leftarrow \mathcal{SDH}(\add T)  \overset{\overset{I_T}\sim}\to \mathcal{SDH}(\mathcal{E}).$$
The isomorphism inverse to $I_T$ is given on the classes of complexes as follows:
$$G: [M] \mapsto  \frac{1}{\prod\limits_{k \in \mathbb{Z}_{\geq 0}} \left\langle M, T_{2k+1} \right\rangle} [\bigoplus\limits_{k \in \mathbb{Z}_{\geq 0}} T_{2k}] \diamond [\bigoplus\limits_{k \in \mathbb{Z}_{\geq 0}} T_{2k+1}]^{-1},$$
where the $T_i$ are the objects in any coresolution of $M$ of the form (\ref{factcores}).
\end{theorem}

We will consider a special case of these tilting equivalences in Section 9.6.


\section{$\mathbb{Z}/2$-graded version of $\mathcal{SDH(E)}$}

In this section, we will develop a $\mathbb{Z}/2$-graded version of all the previous constructions. We should point out once again that the whole concept of the  semi-derived Hall algebras $\mathcal{SDH(E)}$ was inspired by the work \cite{B} of Bridgeland where he considered $\mathbb{Z}/2$-graded complexes with projective components. When we try to generalize this construction to arbitrary exact categories with enough projectives we encounter a major obstacle: the stupid truncation functors $\sigma_{\geq n}$ are no longer defined. Recall that these truncations have been an important ingredient of our construction: we used them in the proofs of Propositions \ref{homfin} and \ref{dbcbequiv} and of Lemmas \ref{homderext}, Lemma \ref{k0inj} and Lemma \ref{qisextclass}.
Another difficulty is that neither the $\mathbb{Z}/2$-graded analogue of Lemma \ref{projres}, nor its analogue for vector bundles seem to work in full generality. Unfortunately, we do not know, at least for the moment, how to resolve all these difficulties for the categories of coherent sheaves; but for the categories $\mathcal{E}$ with enough projectives we find a solution and can construct a $\mathbb{Z}/2$-graded analogue of $\mathcal{SDH(E)}.$ The key point is to replace the category $\mathcal{C}_{\mathbb{Z}/2}(\mathcal{E})$ with a suitable subcategory.

\subsection{Choice of a subcategory}
From now on, we assume that the exact category $\mathcal{E}$ satisfies conditions (C1) and (C2) from Section~3 and the following condition:

\begin{itemize}
\item[(C4')] $\mathcal{E}$ has enough projectives, and each object has a finite projective resolution.
\end{itemize}

Note that (C4') implies condition (C3) from Section~3. We are interested in a subcategory $\mathcal{\widetilde{E}}$ of $\mathcal{C}_{\mathbb{Z}/2}(\mathcal{E}),$ where each object admits a deflation quasi-isomorphism from a complex $P$ with projective components such that, moreover, for each complex $M \in \mathcal{\widetilde{E}},$ we have:
\begin{equation} \label{z2comderext}
\Ext^p_{\mathcal{C}_{\mathbb{Z}/2}(\mathcal{E})} (P, M) = \Ext^p_{\mathcal{D}_{\mathbb{Z}/2}(\mathcal{E})} (P, M), \quad \forall p > 0.
\end{equation}

\begin{definition}
For any exact category $\mathcal{E}$, denote by $\mathcal{\widetilde{E}}$ the closure with respect to extensions and quasi-isomorphism classes of the full subcategory of all stalk complexes inside $\mathcal{C}_{\mathbb{Z}/2}(\mathcal{E}).$ 
\end{definition}

We consider the category $\mathcal{\widetilde{P}}.$ Note that it coincides with the intersection $\mathcal{\widetilde{E}} \cap \mathcal{C}_{\mathbb{Z}/2}(\mathcal{P}).$ Indeed, the inclusion $\mathcal{\widetilde{P}} \subset \mathcal{\widetilde{E}} \cap \mathcal{C}_{\mathbb{Z}/2}(\mathcal{P})$ follows from the definition. The inclusion in the other direction will follow from Proposition \ref{widetildee} below. The following lemmas will be very useful for us.

\begin{lemma} \label{z2acprojres}
For any acyclic complex $K \in \mathcal{C}_{\mathbb{Z}/2, ac}(\mathcal{E}),$ there exists a deflation quasi-isomorphism
$\xymatrix{P_K \ar@{->>}[r]^-{\qis} & K,}$
with $P_K \in \mathcal{\widetilde{P}}.$
\end{lemma}

\begin{proof}
Given an acyclic complex $K = \xymatrix{K^0 \ar@<0.5ex>[r]^{f^0} & K^1 \ar@<0.5ex>[l]^{f^1}} \in \mathcal{C}_{\mathbb{Z}/2}(\mathcal{E}),$
we get two conflations
$$\xymatrix{Z^0 \ar[r]^-{i^0} & K^0 \ar[r]^-{\pi^1} &Z^1, && Z^1 \ar[r]^-{i^1} &K^1 \ar[r]^-{\pi^0} &Z^0.}$$
By (C4'), we have deflations $\xymatrix{P^0 \ar@{->>}[r]^-{p^0} &Z^0}, \xymatrix{P^1 \ar@{->>}[r]^-{p^1} &Z^1},$ with $P^0, P^1 \in \mathcal{P}.$ Now we lift $p^0$ to a morphism $g^0: P^0 \to K^1,$ such that $i^0 \circ g^0 = p^0,$ and similarly we get $g^1: P^1 \to K^0.$ This determines two commutative diagrams:
\begin{equation} \label{pkz}
\xymatrix@R=0.6cm{
*+++{P^1} \ar@{->>}[d] \ar@{>->}[r]  & *+++{P^0 \oplus P^1} \ar@{->}[d] \ar@{->>}[r]  & *+++{P^0} \ar@{->>}[d] & *+++{P^0} \ar@{->>}[d] \ar@{>->}[r]  & *+++{P^0 \oplus P^1} \ar@{->}[d] \ar@{->>}[r]  & *+++{P^1} \ar@{->>}[d] \\
*+++{Z^1}  \ar@{>->}[r]  & *+++{K^0} \ar@{->>}[r]  & *+++{Z^0} & *+++{Z^0}  \ar@{>->}[r]  & *+++{K^1} \ar@{->>}[r]  & *+++{Z^1}\\
}
\end{equation}
By the five-lemma, the two middle vertical arrows are deflations. Also, we get another two commutative diagrams: 
\begin{equation} \label{pk}
\xymatrix@R=0.6cm{
*+++{P^1} \ar@{->>}[d] \ar@{=}[r]  & *+++{P^1} \ar@{->}[d] \ar@{->>}[r]  & *+++{0} \ar@{->>}[d] & *+++{P^0} \ar@{->>}[d] \ar@{=}[r]  & *+++{P^0} \ar@{->}[d] \ar@{->>}[r]  & *+++{0} \ar@{->>}[d] \\
*+++{K^1}  \ar@{>->}[r]  & *+++{K^0} \ar@{->>}[r]  & *+++{K^1} & *+++{K^0}  \ar@{>->}[r]  & *+++{K^1} \ar@{->>}[r]  & *+++{K^0}\\
}
\end{equation}
Consider the complex 
$$P_K = \xymatrix{P^0 \ar@<0.5ex>[r]^{1} & P^0 \ar@<0.5ex>[l]^{0}} \oplus \xymatrix{P^1 \ar@<0.5ex>[r]^{0} & P^1 \ar@<0.5ex>[l]^{1}}.$$
Since diagrams (\ref{pk}) are commutative, we get a deflation $P_K \overset{p'}{\to} K$ that is clearly a quasi-isomorphism. By Lemma \ref{acprojdirsum}, all acyclic complexes with projective components lie in $\mathcal{\widetilde{P}}.$ 
\end{proof}

\begin{lemma} \label{impitilde}
Any $\mathbb{Z}/2-$graded complex in the image $\pi(\mathcal{C}^b(\mathcal{E}))$ can be obtained by a finite sequence of extensions of stalk complexes. In particular, $\pi(\mathcal{C}^b(\mathcal{E}))$ is contained in $\mathcal{\widetilde{E}}.$
\end{lemma}

\begin{proof}
The statement follows from the principle of finite d\'{e}vissage for bounded complexes and from the exactness of the functor $\pi.$
\end{proof}

Now we can verify that $\mathcal{\widetilde{P}}$ and $\mathcal{\widetilde{E}}$ have the desired properties.

\begin{lemma} \label{z2projcool}
For any $P \in \mathcal{\widetilde{P}}, M \in \mathcal{\widetilde{E}},$ the isomorphism (\ref{z2comderext}) holds.
\end{lemma}
 
\begin{proof}
All the arguments of the proof of Proposition \ref{comderext}, except for Lemma \ref{homderext}, work for $P \in \mathcal{C}_{\mathbb{Z}/2}(\mathcal{P}), M \in \mathcal{C}_{\mathbb{Z}/2}(\mathcal{E}).$ It remains to prove that, for any $P \in \mathcal{\widetilde{P}}, M \in \mathcal{\widetilde{E}},$ we have 
\begin{equation} \label{z2homderext}
\Hom_{{\bf K}_{\mathbb{Z}/2}(\mathcal{E})}(P, M) = \Hom_{\mathcal{D}_{\mathbb{Z}/2}(\mathcal{E})}(P, M).
\end{equation}
It is easy to check that for $P, M$ stalk complexes, concentrated in degrees $i$ and $j$ respectively, both sides of (\ref{z2homderext}) are isomorphic to $\Hom_{\mathcal{E}} (P^i, M^j)$ if $i = j$ and vanish if $i \neq j.$ By Lemma \ref{acprojdirsum}, any $P \in \mathcal{C}_{\mathbb{Z}/2, ac}(\mathcal{P})$ is an iterated extension of stalk complexes with projective components, thus, by the five-lemma, the identity (\ref{z2homderext}) holds for any such $P$ and stalk $M.$ In all the following steps we will apply the five-lemma in an appropriate way. Firstly, we prove that the identity (\ref{z2homderext}) holds for stalk complexes $M$ and any $P$ from the closure with respect to extensions and quasi-isomorphisms of the full subcategory of stalk complexes in $\mathcal{C}_{\mathbb{Z}/2}(\mathcal{P}),$ i.e. for any $P \in \mathcal{\widetilde{P}}$ and stalk complex $M.$ Then, using once again Lemma \ref{acprojdirsum}, we prove it for any $P \in \mathcal{\widetilde{P}}, M \in \mathcal{C}_{\mathbb{Z}/2, ac}(\mathcal{P}).$ By Lemma \ref{z2acprojres} and induction, we prove the identity for any $P \in \mathcal{\widetilde{P}}, M \in \mathcal{C}_{\mathbb{Z}/2, ac}(\mathcal{E}).$ Applying the five-lemma once more, we prove it in the desired generality, i.e. for any $P \in \mathcal{\widetilde{P}}, M \in \mathcal{\widetilde{E}}.$ 
\end{proof}
 


\begin{proposition} \label{widetildee}
For any object $A \in \mathcal{\widetilde{E}},$ there exists a deflation quasi-isomorphism 
$P \stackrel{\qis}{\twoheadrightarrow} A,$
with $P \in \mathcal{\widetilde{P}}.$
\end{proposition}

\noindent{\it Proof.\/}
To prove the statement, it is enough to show the following:
\begin{itemize}
\item[(i)] for any $\mathbb{Z}/2-$graded stalk complex there exists such a deflation;
\item[(ii)] the condition of the existence of such a deflation is closed under extensions;
\item[(iii)] this condition is stable under quasi-isomorphism.
\end{itemize}

We first prove (i). Without loss of generality, one can consider a stalk complex of the form  $M^{\bullet} = \xymatrix{M \ar@<0.5ex>[r]^{0} & 0 \ar@<0.5ex>[l]^{0}}.$ Then $\pi(P^{\bullet}) \overset{\qis}{\twoheadrightarrow} M^{\bullet}$ is a desired deflation quasi-isomorphism, for 
$$P^{\bullet} :=\ldots \to 0 \to P^{-d} \to P^{-(d-1)} \to \ldots \to P^0 \to 0 \to \ldots,$$
where
$0 \to P^{-d} \to P^{-(d-1)} \to \ldots \to P^0 \to M \to 0$
is a finite projective resolution of $M$ in $\mathcal{E}.$ Point (iii) is an easy consequence of (ii). Thus, it remains to prove point (ii).


\begin{lemma} \label{z2etildeprojres}
Assume given a conflation $L \rightarrowtail M \twoheadrightarrow N$ in $\mathcal{\widetilde{E}},$
whose end terms $L, N$ admit deflation quasi-isomorphisms
$$P_L \stackrel{\qis}{\twoheadrightarrow} L, \quad\quad P_N \stackrel{\qis}{\twoheadrightarrow} N,$$
with $P_L, P_N \in \mathcal{\widetilde{P}}.$ Then the middle term $M$ also admits a deflation quasi-isomorphism
$P(M) \stackrel{\qis}{\twoheadrightarrow} M,$
with $P(M) \in \mathcal{\widetilde{P}}.$
\end{lemma}

\begin{proof}
Since $\Sigma = \Sigma^{-1},$ we have a conflation 
$\Sigma P_N \rightarrowtail C(1_{\Sigma P_N}) \stackrel{\qis}{\twoheadrightarrow} P_N.$
Since $C(1_{\Sigma P_N})$ is contractible with projective components, it is projective in $\mathcal{\widetilde{P}}.$ In particular, the composition 
$C(1_{\Sigma P_N}) \rightarrowtail P_N \stackrel{\qis}{\twoheadrightarrow} N$ lifts, with respect to the deflation 
$M \twoheadrightarrow N,$ to a morphism $C(1_{\Sigma P_N}) \to M.$ We have an induced map of kernels $\Sigma P_N \stackrel{f}\to L,$
completing the commutative diagram
\begin{equation}
\xymatrix@R=0.6cm{
*+++{\Sigma P_N} \ar@{->}[d] \ar@{>->}[r]  & *+++{C(1_{\Sigma P_N})} \ar@{->}[d] \ar@{->>}[r]  & *+++{P_N} \ar@{->>}[d] \\
*+++{L}  \ar@{>->}[r]  & *+++{M} \ar@{->>}[r]  & *+++{N.} \\
}
\end{equation}
Consider the conflation
$A_L \rightarrowtail P_L \stackrel{\qis}{\twoheadrightarrow} L,$
where $A_L$ is acyclic. Since $P_N$ belongs to $\mathcal{\widetilde{P}},$ so does $\Sigma P_N.$ Thus, 
$$\Ext^1_{\mathcal{C}_{\mathbb{Z}/2}(\mathcal{E})} (\Sigma P_N, A_L) = \Ext^1_{\mathcal{D}_{\mathbb{Z}/2}(\mathcal{E})} (\Sigma P_N, A_L) = 0,$$
and $f$ can be lifted to a morphism $f': \Sigma P_N \to P_L.$ The cone $C(f')$ of the morphism $f'$ coincides with the push-out of the diagram 
$\xymatrix{*+++{P_L} & *+++{\Sigma P_N} \ar@{->}[l]_{f'} \ar@{>->}[r] & *+++{C(1_{\Sigma P_N})}}.$
Thus, we get morphisms $C(f') \to P_N$ and $C(f') \to M$ such that the following diagram commutes: 

\begin{equation} 
\xymatrix@R=0.6cm{
*+++{\Sigma P_N} \ar@{->}[d] \ar@{>->}[r] &*+++{C(1_{\Sigma P_N})} \ar@{->}[d] \ar@{->>}[r] & *+++{P_N} \ar@{=}[d]\\
*+++{P_L} \ar@{->>}[d] \ar@{>->}[r]  & *+++{C(f')} \ar@{->}[d] \ar@{->}[r]  & *+++{P_N} \ar@{->>}[d] \\
*+++{L}  \ar@{>->}[r]  & *+++{M} \ar@{->>}[r]  & *+++{N.} \\
}
\end{equation}
The morphism $C(f') \to M$ is a deflation, by the five-lemma applied to the last two rows. The second row is the conflation induced from the first row by the morphism $f'.$ The category $\mathcal{\widetilde{P}}$ being closed under extensions, $C(f')$ lies in $\mathcal{\widetilde{P}}.$ Since $C(1_{\Sigma P_N})$ is acyclic, we find that the deflation $C(f') \twoheadrightarrow M$ is a quasi-isomorphism. This completes the proof of the lemma and of the proposition.
\end{proof}

Now we can prove the inclusion $\mathcal{\widetilde{E}} \cap \mathcal{C}_{\mathbb{Z}/2}(\mathcal{P}) \subset \mathcal{\widetilde{P}},$ mentioned above. We have just shown that for any $P \in \mathcal{\widetilde{E}},$ in $\mathcal{C}_{\mathbb{Z}/2}(\mathcal{E}),$ there exists a conflation $K \rightarrowtail P' \overset{\qis}{\twoheadrightarrow} P,$ with $P' \in \mathcal{\widetilde{P}}.$ If $P$ has projective components, so does $K.$ Thus, this conflation belongs to $\mathcal{C}_{\mathbb{Z}/2}(\mathcal{P})$ and induces a quasi-isomorphism in this category. But $\mathcal{\widetilde{P}}$ is closed with respect to such quasi-isomorphisms, therefore, $P$ also lies in $\mathcal{\widetilde{P}}.$ We proved the desired inclusion. Thus, $\mathcal{\widetilde{E}} \cap \mathcal{C}_{\mathbb{Z}/2}(\mathcal{P}) = \mathcal{\widetilde{P}}.$

In fact, in the two most motivating cases, the category $\mathcal{\widetilde{E}}$ coincides with the entire category $\mathcal{C}_{\mathbb{Z}/2}(\mathcal{E}).$ 

\begin{proposition} \label{abprojetilde}
Assume that either:
\begin{itemize}
\item{} $\mathcal{E}$ is an abelian category or
\item{} $\mathcal{E}$ is the full exact subcategory of projective objects in a hereditary abelian category.
\end{itemize}
Then any object in $\mathcal{C}_{\mathbb{Z}/2}(\mathcal{E})$ is an iterated extension of stalk complexes. In particular, the category $\mathcal{\widetilde{E}}$ equals $\mathcal{C}_{\mathbb{Z}/2}(\mathcal{E}).$
\end{proposition}

\begin{proof}
The second statement is an immediate consequence of \cite[Lemma\,4.2]{B}. Let now $\mathcal{E}$ be abelian. Consider any complex $\xymatrix{M^0 \ar@<0.5ex>[r]^{d^0} & M^1 \ar@<0.5ex>[l]^{d^1}}$ in $\mathcal{C}_{\mathbb{Z}/2}(\mathcal{E}).$ We have a natural conflation: 
$$\xymatrix@C=0.5cm{*++{(M^0} \ar@<0.5ex>[r]^{d^0} & *++{Z^1)} \ar@<0.5ex>[l]^{0}  \ar@{>->}[r] & *++{(M^0} \ar@<0.5ex>[r]^{d^0} & *++{M^1)} \ar@<0.5ex>[l]^{d^1} \ar@{->>}[r] & *++{(0} \ar@<0.5ex>[r] & *++{{M_1}/{Z^1})} \ar@<0.5ex>[l]},$$
where $Z^1$ is the kernel of $d^1.$
The last term is a stalk complex, the first one is an extension of $\xymatrix{0 \ar@<0.5ex>[r]^{0} & Z^1 \ar@<0.5ex>[l]^{0}}$ by $\xymatrix{M^0 \ar@<0.5ex>[r]^{0} & 0 \ar@<0.5ex>[l]^{0}}.$ 
\end{proof} 

\subsection{Construction of the  algebra $\mathcal{SDH}_{\mathbb{Z}/2}(\mathcal{E})$}

In order to construct $\mathcal{SDH}_{\mathbb{Z}/2}(\mathcal{E}),$ we first define the Euler form 
$$\left\langle \cdot, \cdot \right\rangle: K_0(\mathcal{C}_{\mathbb{Z}/2, ac}(\mathcal{E}) \times K_0(\mathcal{\widetilde{E}}) \to \mathbb{Q}^\times, \quad \left\langle [K], [A] \right\rangle := \prod\limits_{p=0}^{+\infty} |\Ext^{p}_{\mathcal{C}_{\mathbb{Z}/2}(\mathcal{E})} (K, A)|^{(-1)^p}.$$
By the same rule, we define the form
$$\left\langle \cdot, \cdot \right\rangle: K_0(\mathcal{\widetilde{E}}) \times K_0(\mathcal{C}_{\mathbb{Z}/2, ac}(\mathcal{E}) \to \mathbb{Q}^\times.$$
These two forms coincide on $K_0(\mathcal{C}_{\mathbb{Z}/2, ac}(\mathcal{E}) \times K_0(\mathcal{C}_{\mathbb{Z}/2, ac}(\mathcal{E}),$ thus we can denote them by the same symbol and call each of them the Euler form.

\begin{proposition} \label{z2euler}
Each of the Euler forms $\left\langle \cdot, \cdot \right\rangle$ defined above is a well-defined group homomorphism.
\end{proposition}

\begin{proof}
If such an Euler form is well-defined, it is a group homomorphism by the five-lemma; thus, we should verify only that in both of these two alternating products all but a finite number of factors vanish, for all $K, A.$ For this we use Lemma \ref{z2acprojres}: for any given acyclic $K = \xymatrix{K^0 \ar@<0.5ex>[r]^{d^0} & K^1 \ar@<0.5ex>[l]^{d^1}},$ we can find a deflation $P_0 \twoheadrightarrow K,$ with $P_0 \in \mathcal{C}_{\mathbb{Z}/2, ac}(\mathcal{P}).$ As in Lemma \ref{acproj}, we then prove that, 
for each complex $K \in \mathcal{C}_{\mathbb{Z}/2,ac}(\mathcal{E}),$ there exists an acyclic complex of complexes
\begin{equation} \label{z2acprres}
0 \to P_d \to P_{d-1} \to \ldots \to P_1 \to P_0 \to K \to 0,
\end{equation}
where $P_i \in \mathcal{C}_{\mathbb{Z}/2, ac}(\mathcal{P}), i = 0, 1,\ldots, d, \quad d$ is the maximum of the projective dimensions of $K^0$ and $K^1.$ 
As in Corollary \ref{achomfin}, from this it follows that for any $A \in \mathcal{\widetilde{E}},$ we have 
$\Ext^p(K, A) = 0, \quad \forall p > d,$ as desired.
Moreover, as in Lemma \ref{arbproj} and Corollary \ref{acext}, we show that
$$\Ext^p_{\mathcal{C}_{\mathbb{Z}/2}(\mathcal{E})} (A, K) = \Ext^p_{\mathcal{D}_{\mathbb{Z}/2}(\mathcal{E})} (A, K) = 0, \quad \forall p > d + 1,$$
proving the well-definedness of the second form. 
\end{proof}

Define the quantum torus of $\mathbb{Z}/2-$graded acyclic complexes $\mathbb{T}_{\mathbb{Z}/2, ac}(\mathcal{E})$ as the $\mathbb{Q}-$group algebra of $K_0(\mathcal{C}_{\mathbb{Z}/2, ac}(\mathcal{E})),$ with the multiplication twisted by the inverse of the Euler form, i.e. the product of classes of acyclic complexes $K_1, K_2 \in \mathcal{C}_{\mathbb{Z}/2, ac}(\mathcal{E})$ is defined as follows: 
$$[K_1] \diamond [K_2] := \frac{1}{\left\langle K_1, K_2 \right\rangle} [K_1 \oplus K_2].$$ 

\begin{warning}
The canonical map $i': K_0(\mathcal{C}_{\mathbb{Z}/2, ac}(\mathcal{E})) \to K_0 (\mathcal{\widetilde{E}})$ is not injective. E.g., we have
\begin{equation} \label{notinj}
i'([\xymatrix{X \ar@<0.5ex>[r]^{1} & X \ar@<0.5ex>[l]^{0}}]) = [\xymatrix{X \ar@<0.5ex>[r]^{0} & 0 \ar@<0.5ex>[l]^{0}}] + [\xymatrix{0 \ar@<0.5ex>[r]^{0} & X \ar@<0.5ex>[l]^{0}}] = i'([\xymatrix{X \ar@<0.5ex>[r]^{0} & X \ar@<0.5ex>[l]^{1}}]).
\end{equation}
Cf. the proof of Theorem \ref{z2assocfree}. Moreover, the Euler form cannot be defined as before on the whole product $K_0(\mathcal{\widetilde{E}}) \times K_0(\mathcal{\widetilde{E}}),$ if $\mathcal{E}$ has non-trivial Euler form. 
\end{warning}

Indeed, assume that we have a group homomorphism 
$$\left\langle \cdot, \cdot \right\rangle: K_0(\mathcal{\widetilde{E}}) \times K_0(\mathcal{\widetilde{E}}) \to \mathbb{Q}^\times, \quad \left\langle [K], [A] \right\rangle := \prod\limits_{p=0}^{+\infty} |\Ext^{p}_{\mathcal{C}_{\mathbb{Z}/2}(\mathcal{E})} (K, A)|^{(-1)^p}.$$
If $\left\langle Z, Y \right\rangle_{\mathcal{E}} \neq 0,$ for certain $Z, Y \in \mathcal{E},$ it is easy to check that 
$$\left\langle [\xymatrix{Z \ar@<0.5ex>[r]^{1} & Z \ar@<0.5ex>[l]^{0}}], [\xymatrix{Y \ar@<0.5ex>[r]^{1} & 0 \ar@<0.5ex>[l]^{0}}] \right\rangle = \left\langle Z, Y \right\rangle_{\mathcal{E}} \neq 0 = \left\langle [\xymatrix{Z \ar@<0.5ex>[r]^{0} & Z \ar@<0.5ex>[l]^{1}}], [\xymatrix{Y \ar@<0.5ex>[r]^{1} & 0 \ar@<0.5ex>[l]^{0}}] \right\rangle,$$ 
but the classes of $\xymatrix{Z \ar@<0.5ex>[r]^{1} & Z \ar@<0.5ex>[l]^{0}}$ and $\xymatrix{Z \ar@<0.5ex>[r]^{0} & Z \ar@<0.5ex>[l]^{1}}$ in $K_0(\mathcal{\widetilde{E}})$ coincide, cf. (\ref{notinj}).
Here by $\left\langle \cdot, \cdot \right\rangle_{\mathcal{E}}$ we mean the Euler form of the category $\mathcal{E}.$
\smallskip\par

Define the {\it quantum affine space of} $\mathbb{Z}/2-${\it graded acyclic complexes} $\mathbb{A}_{\mathbb{Z}/2, ac}(\mathcal{E})$ as the $\mathbb{Q}-$monoid algebra of the Grothendieck monoid $M_0(\mathcal{C}_{\mathbb{Z}/2, ac}(\mathcal{E})),$ with the multiplication twisted by the inverse of the Euler form.
Define the {\it left relative Grothendieck monoid} $M_{0}' (\mathcal{\widetilde{E}})$ as the free monoid generated by the set $\Iso(\mathcal{\widetilde{E}}),$ divided by the following set of relations:
$$\left\langle [L] = [K \oplus M] | K \rightarrowtail L \twoheadrightarrow M \quad \mbox{is a conflation}, K \in \mathcal{C}_{\mathbb{Z}/2,ac}(\mathcal{E}) \right\rangle.$$ 
Similarly, define the {\it left relative Grothendieck group} $K_{0}' (\mathcal{\widetilde{E}})$ as the free group generated by the set $\Iso(\mathcal{\widetilde{E}}),$ divided by the following set of relations:
$$\left\langle [K] - [L] + [M] = 0 | K \rightarrowtail L \twoheadrightarrow M \quad \mbox{is a conflation}, K \in \mathcal{C}_{\mathbb{Z}/2,ac}(\mathcal{E}) \right\rangle.$$
Consider the $\mathbb{Q}-$vector space with basis parametrized by the elements of $M_{0}' (\mathcal{\widetilde{E}}).$ Define on this space a structure of a  bimodule $\mathcal{M'}_{\mathbb{Z}/2}(\mathcal{E})$ over $\mathbb{A}_{\mathbb{Z}/2, ac}(\mathcal{E})$ by the rule 
$$[K] \diamond [M] := \frac{1}{\left\langle K, M \right\rangle} [K \oplus M], \quad\quad [M] \diamond [K] := \frac{1}{\left\langle M, K \right\rangle} [M \oplus K]$$
for $K \in \mathcal{C}_{\mathbb{Z}/2, ac}(\mathcal{E}), M \in \mathcal{\widetilde{E}}.$ Then $\mathcal{M}_{\mathbb{Z}/2}(\mathcal{E}) := \mathbb{T}_{\mathbb{Z}/2, ac}(\mathcal{E}) \otimes_{\mathbb{A}_{\mathbb{Z}/2, ac}(\mathcal{E})} \mathcal{M'}_{\mathbb{Z}/2}(\mathcal{E}) \otimes_{\mathbb{A}_{\mathbb{Z}/2, ac}(\mathcal{E})} \mathbb{T}_{\mathbb{Z}/2, ac}(\mathcal{E})$ is a bimodule over the quantum torus $\mathbb{T}_{\mathbb{Z}/2, ac}(\mathcal{E}).$ 

\begin{theorem} \label{z2free}
$\mathcal{M}_{\mathbb{Z}/2}(\mathcal{E})$ is a free right module over $\mathbb{T}_{\mathbb{Z}/2, ac}(\mathcal{E}).$ Each choice of representatives of the quasi-isomorphism classes yields a basis.
\end{theorem}

\begin{proof}
The proof of Theorem \ref{zfree} would work if Lemma \ref{k0inj} did not fail: the natural map $i': K_0(\mathcal{C}_{\mathbb{Z}/2, ac}(\mathcal{E})) \to K_0 (\mathcal{C}_{\mathbb{Z}/2}(\mathcal{E}))$ is not injective; cf. the above warning. 
Note that nonetheless, this injectivity has been used only to prove that there is a grading on the whole module compatible with the grading of any quasi-isomorphism component $\mathcal{M}_{\alpha} (\mathcal{E})$ (if we fix a representative $E$ such that $\overline{E} = \alpha$) by $K_0(\mathcal{C}^b_{ac}(\mathcal{E})).$ Note that the grading we have chosen in the $\mathbb{Z}-$graded case can be now refined: indeed, $\mathcal{M}_{\mathbb{Z}/2}(\mathcal{E})$ is naturally graded by $K_0'(\mathcal{\widetilde{E}}).$ The following lemma provides the compatibility of the $K_0(\mathcal{C}_{\mathbb{Z}/2, ac}(\mathcal{E}))-$grading of any quasi-isomorphism component with this one.

\begin{lemma} \label{k0'inj}
The natural map 
$$i: K_0(\mathcal{C}_{\mathbb{Z}/2, ac}(\mathcal{E})) \rightarrow K_0'(\mathcal{\widetilde{E}}), [M] \mapsto [M]$$
is injective.
\end{lemma}

\begin{proof}
We will define a map 
$\Phi: K_0'(\mathcal{\widetilde{E}}) \rightarrow K_0(\mathcal{C}_{\mathbb{Z}/2, ac}(\mathcal{E})),$
such that 
\begin{equation} \label{phii}
\Phi \circ i = \mbox{Id}_{K_0(C_{\mathbb{Z}/2, ac}(\mathcal{E}))}.
\end{equation}
We start by constructing such a map for complexes with projective components, i.e. we construct a retraction $\Phi_{\mathcal{P}}$ for 
$$i_{\mathcal{P}}: K_0(\mathcal{C}_{\mathbb{Z}/2, ac} (\mathcal{P})) \rightarrow K_0'(\mathcal{\widetilde{P}}), [M] \mapsto [M].$$
Since acyclic complexes with projective components are projective objects in $\mathcal{\widetilde{P}},$ all relations in $K_0'(\mathcal{\widetilde{P}})$ come actually from direct sums. This yields a natural projection map 
$$p: K_0'(\mathcal{\widetilde{P}}) \twoheadrightarrow K_0^{split}(\mathcal{\widetilde{P}}), [M] \mapsto [M].$$
Since $\mathcal{P}$ is Krull-Schmidt category, so is $\mathcal{C}_{\mathbb{Z}/2}(\mathcal{P})$ and, therefore, $\mathcal{\widetilde{P}}.$ Therefore, one can define an ``acyclic part'' of a complex with projective components: each object $M \in \mathcal{\widetilde{P}}$ can be decomposed in a unique way (up to a permutation of factors) into a finite direct sum of indecomposable complexes:
$$M = \bigoplus\limits_{i=1}^{m(M)} M_i \oplus \bigoplus\limits_{j=1}^{k(M)} M_j',$$
where all $M_i$ are acyclic while the $M_j'$ are not. Then 
$$\phi: K_0^{split}(\mathcal{\widetilde{P}}) \twoheadrightarrow K_0^{split}(\mathcal{C}_{\mathbb{Z}/2, ac}(\mathcal{P})) = K_0(\mathcal{C}_{\mathbb{Z}/2, ac}(\mathcal{P})), [M] \mapsto \bigoplus\limits_{i=1}^{m(M)} M_i$$
is a well-defined group epimorphism, and for $\Phi_{\mathcal{P}} := \phi \circ p$ we get 
$$\Phi_{\mathcal{P}} \circ i_{\mathcal{P}} = \mbox{Id}_{K_0(\mathcal{C}_{\mathbb{Z}/2, ac}(\mathcal{P}))}.$$

Given any complex $M \in \mathcal{\widetilde{E}},$ by Proposition \ref{widetildee}, we have at least one deflation quasi-isomorphism $P_M \stackrel{\qis}{\twoheadrightarrow} M,$ with $P_M \in \mathcal{\widetilde{P}};$  consider also the kernel of this deflation $A_M.$ We define the left inverse to the inclusion $i$ as 
$$\Phi([M]) := [\Phi_{\mathcal{P}}(P_M)] - [A_M].$$
We should verify two things: 
\begin{itemize}
\item[1)] The element $\Phi(M)$ thus defined does not depend on the choice of a deflation quasi-isomorphism $P_M \stackrel{\qis}{\twoheadrightarrow} M;$
\item[2)] The map $\Phi$ is additive on conflations $K \rightarrowtail L \twoheadrightarrow M$ of $\widetilde{E}$ with $K$ acyclic, i.e. it is actually a well-defined group homomorphism.
\end{itemize}
We prove part 1) in the same manner as Proposition \ref{welldef} concerning the well-definedness of the multiplication in $\mathcal{SDH(E)}.$ Namely, using suitable pull-backs, we reduce the problem to the case where we have two resolutions of $M$ linked by a deflation quasi-isomorphism $P_M \stackrel{\qis}{\twoheadrightarrow} {P'}_M.$ For these two resolutions, we construct a commutative $3 \times 3-$diagram, similar to diagram (\ref{diag1}):

\begin{equation} \label{diagphiwell}
\xymatrix@R=0.6cm{
*+++{H} \ar@{>->}[d] \ar@{=}[r] &*+++{H} \ar@{>->}[d] \ar@{->>}[r] & *+++{0} \ar@{>->}[d]\\
*+++{N} \ar@{->>}[d] \ar@{>->}[r]  & *+++{P_M} \ar@{->>}[d] \ar@{->>}[r]  & *+++{M} \ar@{=}[d] \\
*+++{K}  \ar@{>->}[r]  & *+++{P'_M} \ar@{->>}[r]  & *+++{M.} \\
}
\end{equation}
Here $H, N, K$ are acyclic complexes. By considering the second column we see that $H$ has projective components. Since $\Phi_{\mathcal{P}}$ is well-defined on $K_{0}'(\mathcal{P}),$ we have $\Phi_{\mathcal{P}}([P_M]) - \Phi_{\mathcal{P}}([P'_M]) = \Phi_{\mathcal{P}}([H]).$ But $\Phi_{\mathcal{P}}([H]) = [H] = [N] - [K],$ and it follows that 
$$\Phi_{\mathcal{P}}([P_M)]) - [N] = \Phi_{\mathcal{P}}([P'_M]) - [K],$$
proving the desired independence.

Let us prove the additivity of the map $\Phi.$ We should show that for any conflation $K \rightarrowtail L \twoheadrightarrow M,$ with $K$ acyclic, we have $\Phi(K) - \Phi(L) + \Phi(M) = 0.$ Take a pair of conflations $A_K \rightarrowtail P_K \stackrel{\qis}{\twoheadrightarrow} K$ and $A_M \rightarrowtail P_M \stackrel{\qis}{\twoheadrightarrow} M.$ Since $\Ext^1(P_M, K) = 0$ (by Lemma \ref{z2projcool}), the morphism $P_M \twoheadrightarrow M$ can be lifted along the deflation $L \twoheadrightarrow M$ and we can find a deflation quasi-isomorphism $P_K \oplus P_M \stackrel{\qis}{\twoheadrightarrow} L$ with the kernel $A_L.$ With the induced conflation of kernels $A_K \rightarrowtail A_L \twoheadrightarrow A_M,$ we get the commutative diagram

\begin{equation}
\xymatrix@R=0.6cm{
*+++{A_K} \ar@{>->}[d] \ar@{>->}[r] &*+++{A_L} \ar@{>->}[d] \ar@{->>}[r] & *+++{A_M} \ar@{>->}[d]\\
*+++{P_K} \ar@{->>}[d] \ar@{>->}[r]  & *+++{P_K \oplus P_M} \ar@{->>}[d] \ar@{->>}[r]  & *+++{P_M} \ar@{->>}[d] \\
*+++{K}  \ar@{>->}[r]  & *+++{L} \ar@{->>}[r]  & *+++{M.} \\
}
\end{equation}
Using the well-definedness of $\Phi$ and the additivity of $\Phi_{\mathcal{P}},$ it is easy to check now that 
\newline$\Phi(K) - \Phi(L) + \Phi(M) = 0.$

It remains to check identity (\ref{phii}). Take an acyclic complex $M \in C_{\mathbb{Z}/2, ac}(\mathcal{E}).$ By Lemma \ref{z2acprojres}, there exists a deflation quasi-isomorphism $P_M  \stackrel{\qis}{\twoheadrightarrow} M$ with kernel $A_M.$ Since $M$ is acyclic, so is $P_M.$ Therefore, $\Phi_{\mathcal{P}}(P_M) = [P_M]$ and 
$$\Phi \circ i ([M]) = \Phi(M) = [P_M] - [A_M] = [M].$$
By the additivity proved above, we get the desired identity of homomorphisms.
\end{proof}

With these two compatible gradings at hand, we can mimic the proof of Theorem \ref{zfree}, replacing everywhere bounded complexes by the $\mathbb{Z}/2-$graded and $K_0(\mathcal{C}^b(\mathcal{E}))$ by $K_0(\mathcal{\widetilde{E}}).$
\end{proof}

\begin{definition}
We endow $\mathcal{M}_{\mathbb{Z}/2}(\mathcal{E})$ with the following multiplication: the product of the classes of two complexes $L, M \in \mathcal{\widetilde{E}}$ is defined as 
\begin{equation}
[L] \diamond [M] = \frac{1}{\left\langle A_L, L \right\rangle} [A_L]^{-1} \diamond \sum\limits_{\varepsilon \in \Ext^1_{\mathcal{\widetilde{E}}}(P_L, M)} \frac{[\mt(\varepsilon)]}{|\Hom(P_L, M)|},
\end{equation}
or, equivalently, as
\begin{equation}
[L] \diamond [M] = \frac{1}{\left\langle A_L, L \right\rangle} [A_L]^{-1} \diamond \sum\limits_{X \in \Iso (\widetilde{\mathcal{E}})} (\frac{|\Ext^1_{\widetilde{\mathcal{E}}}(P_L, M)_{X}|}{|\Hom(P_L, M)|} [X]),
\end{equation}
where 
$A_L \rightarrowtail P_L \stackrel{\qis}{\twoheadrightarrow} L $
is an arbitrary conflation with $A_L \in \mathcal{C}_{\mathbb{Z}/2, ac}(\mathcal{E}), P_L \in \mathcal{\widetilde{P}}.$ We call the resulting algebra the {\it $\mathbb{Z}/2-$graded semi-derived Hall algebra} $\mathcal{SDH}_{\mathbb{Z}/2}(\mathcal{E}).$
\end{definition}

\begin{proposition}
The multiplication $\diamond$ of $\mathcal{SDH}_{\mathbb{Z}/2}(\mathcal{E})$ is well-defined and compatible with the module structure.
\end{proposition}

\begin{proof}
The proof follows the lines of the proofs of Propositions \ref{welldef} and \ref{zprodcomp}.
\end{proof}

We do not know, whether the $\mathbb{Z}/2-$graded versions of Lemma \ref{qisextclass}, Corollary \ref{qisextclassident} and Corollary \ref{qisextclassmult} do hold, since their proofs use Lemma \ref{k0inj}.

In \cite{B}, Bridgeland uses a non-trivial twist. Let us construct a similarly twisted version of the algebra $\mathcal{SDH}_{\mathbb{Z}/2}(\mathcal{E}).$ 
We define a bilinear form 
$$\left\langle \cdot, \cdot \right\rangle_{\cw}: \Iso(\mathcal{C}_{\mathbb{Z}/2}(\mathcal{E})) \times \Iso(\mathcal{C}_{\mathbb{Z}/2}(\mathcal{E})) \to \mathbb{Q}^\times, \quad \left\langle M, N \right\rangle_{\cw} := \sqrt{\left\langle M^0, N^0 \right\rangle \cdot \left\langle M^1,N^1 \right\rangle},$$
where $M$ has the form $\xymatrix{M^0 \ar@<0.5ex>[r] & M^1 \ar@<0.5ex>[l]},$ and similarly for $N.$
This form descends to a bilinear form (denoted by the same symbol)
$$\left\langle \cdot, \cdot \right\rangle_{\cw}: A \times B \to \mathbb{Q}^\times,$$
where $A$ and $B$ are some of the Grothendieck groups considered above, i.e.
$$A, B \in \left\{K_0(\mathcal{C}_{\mathbb{Z}/2}(\mathcal{E})), K_0(\mathcal{C}_{\mathbb{Z}/2, ac}(\mathcal{E})), K_0(\widetilde{\mathcal{E}}), K_0'(\mathcal{\widetilde{E}}) \right\}.$$

\begin{lemma}
On the Grothendieck group of acyclic complexes, the bilinear form $\left\langle \cdot, \cdot \right\rangle_{\cw}$ coincides with the usual Euler form:
$$\left\langle \alpha, \beta \right\rangle_{\cw} = \left\langle \alpha, \beta \right\rangle, \quad \forall (\alpha, \beta) \in K_0(\mathcal{C}_{\mathbb{Z}/2, ac}(\mathcal{E})) \times K_0(\mathcal{C}_{\mathbb{Z}/2, ac}(\mathcal{E})).$$
\end{lemma}

\begin{proof}
The statement is trivial for $\alpha, \beta \in \left\{[K_P], [K_P^*] | P \in \mathcal{P} \right\}.$ We know (cf. the proof of Proposition \ref{z2euler}) that these classes generate the whole Grothendieck group $K_0(\mathcal{C}_{\mathbb{Z}/2, ac}(\mathcal{E})).$ 
\end{proof}

We define the {\it twisted $\mathbb{Z}/2-$graded semi-derived Hall algebra} $\mathcal{SDH}_{\mathbb{Z}/2,tw}(\mathcal{E})$ in the same way as the non-twisted version, but with replacing $\mathbb{A}_{\mathbb{Z}/2,ac}(\mathcal{E})$ and $\mathbb{T}_{\mathbb{Z}/2,ac}(\mathcal{E})$ by the monoid and the group algebras $\mathbb{Q} M_0'(\mathcal{C}_{\mathbb{Z}/2, ac}(\mathcal{E}))$ and $\mathbb{Q} K_0(\mathcal{C}_{\mathbb{Z}/2, ac}(\mathcal{E})),$ respectively. The multiplication in $\mathcal{SDH}_{\mathbb{Z}/2,tw}(\mathcal{E})$ is given by the rule 
$$[M_1] * [M_2] := \left\langle M_1, M_2 \right\rangle_{\cw} [M_1] \diamond [M_2], \quad [M_1], [M_2] \in K_0'(\widetilde{\mathcal{E}}).$$
It is easy to check that this is a free bimodule over $\mathbb{Q} K_0(\mathcal{C}_{\mathbb{Z}/2, ac}(\mathcal{E})),$ with the same type of bases as in theorem \ref{z2free}.
We define the {\it reduced twisted} version of the $\mathbb{Z}/2-$graded semi-derived Hall algebra by setting $[K] = 1$ whenever $K$ is an acyclic
complex, invariant under the shift functor:

$$\mathcal{SDH}_{\mathbb{Z}/2,tw,red}(\mathcal{E}) := \mathcal{SDH}_{\mathbb{Z}/2,tw}(\mathcal{E})/([K] - 1: K \in \mathcal{C}_{\mathbb{Z}/2,ac}(\mathcal{E}), K \cong K^*).$$
Due to isomorphism (\ref{k0k0}), this is the same as setting

\begin{equation} \label{reduction}
K_\alpha * K_\alpha^* = 1, \quad \forall \alpha \in K_0(\mathcal{E}),
\end{equation}
where
$$
K_\alpha := [{{K_A}}] * [{{K_B}}]^{-1},\quad
K_\alpha^* := [{{K_A}}^*] * [{{K_B}}^*]^{-1},
$$
for any $A, B \in \mathcal{E}$ such that $\alpha = A - B.$ It is easy to check (cf. \cite{B}) that these two elements are well-defined, i.e. they do not depend on choice of $A, B.$ 

\subsection{Associativity and derived invariance}

\begin{theorem} \label{z2assocfree}
$\mathcal{SDH}_{\mathbb{Z}/2}(\mathcal{E})$ is an associative unital algebra.
\end{theorem}

\begin{proof}
The class $[0]$ is clearly the unit of $\mathcal{SDH}_{\mathbb{Z}/2}(\mathcal{E}).$ One can check that the proof of Theorem \ref{zassoc} concerning the associativity works in this $\mathbb{Z}/2-$graded case without any changes. 
\end{proof}

\begin{theorem} \label{equivsdh2}
Suppose that $F: \mathcal{E'} \to \mathcal{E}$ is an exact functor between exact categories satisfying conditions (C1), (C2) and (C4'), inducing an equivalence of bounded derived categories
$$F: \mathcal{D}^b(\mathcal{E'}) \overset\sim\to \mathcal{D}^b(\mathcal{E})$$
Then $F$ induces an isomorphism in the $\mathbb{Z}/2-$graded semi-derived Hall algebras:
$$F: \mathcal{SDH}_{\mathbb{Z}/2}(\mathcal{E'}) \overset\sim\to \mathcal{SDH}_{\mathbb{Z}/2}(\mathcal{E}).$$
\end{theorem}

\begin{proof}
Since $F$ induces an equivalence of the bounded derived categories, it induces an isomorphism of their Grothendieck groups and preserves the Euler form. But $K_0(\mathcal{D}^b(\mathcal{E})) = K_0 (\mathcal{E}),$ hence $F$ induces an isomorphism $K_0 (\mathcal{E'}) \overset\sim\to K_0 (\mathcal{E}).$ It easily follows from the existence of resolutions (\ref{z2acprres}), that $K_0(\mathcal{C}_{\mathbb{Z}/2, ac}(\mathcal{P})) \overset\sim\to K_0(\mathcal{C}_{\mathbb{Z}/2, ac}(\mathcal{E})).$ By Lemma \ref{acprojdirsum} and because the exact structure of $K_0(\mathcal{C}_{\mathbb{Z}/2, ac}(\mathcal{P}))$ is split, we have a canonical isomorphism 
$$K_0(\mathcal{C}_{\mathbb{Z}/2, ac}(\mathcal{P})) \overset\sim\to  K_0 (\mathcal{P}) \oplus K_0 (\mathcal{P}).$$
But $K_0 (\mathcal{P}) \overset\sim\to K_0 (\mathcal{E})$ by our assumptions. By composing all these isomorphisms, we finally find that $K_0(\mathcal{C}_{\mathbb{Z}/2, ac}(\mathcal{E})) \overset\sim\to  K_0 (\mathcal{E}) \oplus K_0 (\mathcal{E}).$
In other words, the natural homomorphism:
\begin{equation} \label{k0k0}
K_0 (\mathcal{E}) \oplus K_0 (\mathcal{E}) \to K_0(\mathcal{C}_{\mathbb{Z}/2, ac}(\mathcal{E})), \quad([M],[N]) \mapsto [\xymatrix{M \ar@<0.5ex>[r]^{1} & M \ar@<0.5ex>[l]^{0}}] + [\xymatrix{N \ar@<0.5ex>[r]^{0} & N \ar@<0.5ex>[l]^{1}}]
\end{equation}
is an isomorphism, and similarly for $\mathcal{E'}.$
These two isomorphisms are compatible with $F;$ therefore, $F$ induces an isomorphism of the Grothendieck groups of acyclic $\mathbb{Z}/2-$graded complexes. It clearly preserves the Euler form. It follows that $F$ induces an isomorphism
$$\mathbb{T}_{\mathbb{Z}/2, ac}(\mathcal{E'}) \overset\sim\to \mathbb{T}_{\mathbb{Z}/2, ac}(\mathcal{E}).$$

Since $F$ induces an equivalence $\mathcal{D}^b(\mathcal{E}') \overset\sim\to \mathcal{D}^b(\mathcal{E}),$ it induces an equivalence between the images of the bounded derived categories under the natural map $\pi$ into the $\mathbb{Z}/2-$graded derived categories and also an equivalence of the extension closures of these images.
Let $\mathcal{D}^{'}_{\mathbb{Z}/2}(\mathcal{E})$ denote this extension closure of $\pi(\mathcal{D}^b(\mathcal{E})),$ and similarly for $\mathcal{E'}.$ Note that stalk complexes lie in the image of $\pi$ and, therefore, in $\mathcal{D}^{'}_{\mathbb{Z}/2}(\mathcal{E'})$ and $\mathcal{D}^{'}_{\mathbb{Z}/2}(\mathcal{E}),$ respectively. We have natural fully faithful functors 
$$\psi_\mathcal{E}: \mathcal{\widetilde{E}} [{\qis}^{-1}] \to \mathcal{D}^{'}_{\mathbb{Z}/2}(\mathcal{E}), \quad \psi_\mathcal{E'}: \mathcal{\widetilde{E'}} [{\qis}^{-1}] \to \mathcal{D}^{'}_{\mathbb{Z}/2}(\mathcal{E'}),$$
which are both the identity on objects (in fact, both of them are equivalences). 
Therefore, $F$ induces a bijection between the sets of quasi-isomorphism classes of objects in $\mathcal{\widetilde{E}}$ and $\mathcal{\widetilde{E'}}$ i.e., by Theorem \ref{z2free}, between the bases of $\mathcal{SDH}_{\mathbb{Z}/2}(\mathcal{E'})$ and $\mathcal{SDH}_{\mathbb{Z}/2}(\mathcal{E})$ as modules over the (isomorphic to each other) quantum tori. By property (\ref{piext}) and Proposition \ref{equivdbcb}, all homomorphism and extension spaces in $\mathcal{\widetilde{E}}$ are preserved under $F.$ By the same reasons, the Euler form on $K_0(\mathcal{C}_{\mathbb{Z}/2, ac}(\mathcal{E})) \times K_0(\mathcal{\widetilde{E}})$ is preserved by $F.$ Since $\psi_\mathcal{E'}$ and $\psi_\mathcal{E}$ are equivalences and $F$ induces an equivalence between $\mathcal{D}^{'}_{\mathbb{Z}/2}(\mathcal{E'})$ and $\mathcal{D}^{'}_{\mathbb{Z}/2}(\mathcal{E}),$ the multiplication (in particular, the action of the quantum torus) in $\mathcal{SDH}_{\mathbb{Z}/2}(\mathcal{E'})$ is preserved under $F,$ hence $F: \mathcal{SDH}_{\mathbb{Z}/2}(\mathcal{E'}) \overset\sim\to \mathcal{SDH}_{\mathbb{Z}/2}(\mathcal{E}).$
\end{proof}

\begin{theorem} \label{z2sdhequivtw}
The algebras $\mathcal{SDH}_{\mathbb{Z}/2,tw}(\mathcal{E})$ and $\mathcal{SDH}_{\mathbb{Z}/2,tw,red}(\mathcal{E})$ are associative and unital. Each derived equivalence $F$ from Theorem \ref{equivsdh2} induces an isomorphism of the twisted and reduced twisted $\mathbb{Z}/2-$graded semi-derived Hall algebras.
\end{theorem}

\begin{proof}
The first part follows from Theorem \ref{z2assocfree} and the well-definedness of the form $\left\langle \cdot, \cdot \right\rangle_{\cw}$ as a homomorphism $K_0'(\widetilde{\mathcal{E}}) \times K_0'(\widetilde{\mathcal{E}}) \to \mathbb{Q}^\times.$ To prove the invariance of $\mathcal{SDH}_{\mathbb{Z}/2,tw}(\mathcal{E})$, it is enough to prove that this form is preserved under the functor $F.$ Since $F$ is induced by an exact functor between exact categories, it sends the $i-$th component of a complex in $\widetilde{\mathcal{E'}}$ to the $i-$th component of its image in $\widetilde{\mathcal{E}},$ for $i = 0, 1.$ The Euler form on the Grothendieck group $K_0(\mathcal{E'})$ is sent by $F$ to the Euler form on $K_0(\mathcal{E}),$ since all extension spaces are certain homomorphism spaces in (equivalent) derived categories. By these two arguments, the form $\left\langle \cdot, \cdot \right\rangle_{\cw}$ is preserved under $F.$ 
Isomorphism (\ref{k0k0}) being compatible with $F,$ so is condition (\ref{reduction}). 
\end{proof}

As in the $\mathbb{Z}-$graded case, we get the following corollaries of Theorems \ref{equivsdh2} and \ref{z2sdhequivtw}.

\begin{corollary} \label{sdh2pequiv}
We have the following isomorphisms of algebras:
\begin{equation} \label{allprojz2}
\mathcal{H}(\widetilde{\mathcal{P}})[[K]^{-1}| H^{\bullet}(K) = 0] \cong \mathcal{SDH}_{\mathbb{Z}/2}(\mathcal{P}) \overset{I}{\overset\sim\to} \mathcal{SDH}_{\mathbb{Z}/2}(\mathcal{E});
\end{equation}
where the isomorphism $I$ is induced by the inclusion functor $\mathcal{P} \hookrightarrow \mathcal{A}.$ These isomorphisms are compatible with the twist and the reduction considered above. The isomorphism $F,$ inverse to $I,$ is defined as in Proposition \ref{sdhpequiv}.
\end{corollary}

\begin{corollary} \label{z2tiltequiv}
Let $T$ be a tilting object in an exact category $\mathcal{E}$ satisfying conditions (C1), (C2) and (C4'). We have the following isomorphisms between the $\mathbb{Z}/2$-graded semi-derived Hall algebras:
$$\mathcal{SDH}_{\mathbb{Z}/2}(\mod(\End(T))) \overset\sim\leftarrow \mathcal{SDH}_{\mathbb{Z}/2}(\add T) \overset{\overset{I_T}\sim}\to \mathcal{SDH}_{\mathbb{Z}/2}(\mathcal{E}),$$
where $I_T$ is induced by the inclusion functor
$I_T: \add T \hookrightarrow \mathcal{E}.$ These isomorphisms are compatible with the twist and the reduction considered above.
The isomorphism $G,$ inverse to $I_T,$ is defined as in Theorem \ref{tiltequiv}. 
\end{corollary}




\subsection{Hereditary case: Bridgeland's construction, Drinfeld doubles and quantum groups}

As in the $\mathbb{Z}-$graded case, we have an alternative formula for the product in $\mathcal{SDH}_{\mathbb{Z}/2}(\mathcal{E}),$ if $\mathcal{E}$ is hereditary. The proof is the same as for Theorem \ref{hermult}.

\begin{theorem} \label{z2hermult}
Assume that $\mathcal{E}$ is hereditary and has enough projectives. Then for $L, M \in \widetilde{\mathcal{E}},$ the product $[L] \diamond [M]$ is equal to the following sum:
\begin{equation}
[L] \diamond [M] = \sum\limits_{X \in \Iso(\widetilde{\mathcal{E}})} \frac{|\Ext^1_{\widetilde{\mathcal{E}}} (L, M)_X|}{|\Hom(L, M)|} [X],
\end{equation}
\end{theorem}

\begin{corollary} \label{z2herquotloc}
Assume that $\mathcal{E}$ is hereditary and has enough projectives. Then there exists an algebra homomorphism 
$$p: \mathcal{H}(\widetilde{\mathcal{E}}) \to \mathcal{SDH_{\mathbb{Z}/2}(E)}.$$
The homomorphism $p$ induces an algebra isomorphism
\begin{equation}
(\mathcal{H}(\widetilde{\mathcal{E}})/I_{\mathbb{Z}/2})[S_{\mathbb{Z}/2}^{-1}] \overset\sim\to \mathcal{SDH_{\mathbb{Z}/2}(E)},
\end{equation}
where $I_{\mathbb{Z}/2}$ is the two-sided ideal generated by all differences
$[L] - [K\oplus M]$, where $K \rightarrowtail L \twoheadrightarrow M$ is a conflation in $\widetilde{\mathcal{E}}$ with acyclic $K,$ and $S_{\mathbb{Z}/2}$ is the set of all classes $[K]$ of acyclic $\mathbb{Z}/2-$graded complexes.
\end{corollary}

\begin{corollary} \label{j+-}
Assume that $\mathcal{E}$ is hereditary and has enough projectives.  There is an embedding of algebras 
$$
J_+^{e}: \mathcal{H}_{tw}^{e}(\mathcal{E}) \hookrightarrow \mathcal{SDH}_{\mathbb{Z}/2,tw}(\mathcal{E})
$$
defined by
$$
 [A] \longmapsto [\xymatrix{0 \ar@<0.5ex>[r] & A \ar@<0.5ex>[l]}] , \quad
 K_{\alpha} \longmapsto K_\alpha,
$$
where $A \in \mathcal{E}, \alpha \in K_0(\mathcal{E}).$
By composing $J_+^{e}$ and the involution $*$, 
we also have an embedding 
$$
J_{-}^{e}: \mathcal{H}_{tw}^{e}(\mathcal{E}) \hookrightarrow \mathcal{SDH}_{\mathbb{Z}/2,tw}(\mathcal{E})
$$
defined by 
$$
 [A] \longmapsto [\xymatrix{A \ar@<0.5ex>[r] & 0 \ar@<0.5ex>[l]}], \quad
 K_{\alpha} \longmapsto K_\alpha^*.
$$
\end{corollary}

Let now $\mathcal{A}$ be a hereditary abelian category satisfying conditions (C1) and (C2) and having enough projectives. By Proposition \ref{abprojetilde}, we have $\widetilde{\mathcal{A}} = \mathcal{C}_{\mathbb{Z}/2}(\mathcal{A});$ similarly, we have $\widetilde{\mathcal{P}} = \mathcal{C}_{\mathbb{Z}/2}(\mathcal{P}).$ By Corollary \ref{sdh2pequiv}, we have an isomorphism 
\begin{equation} \label{dh(a)}
I: \mathcal{H}_{tw}(\mathcal{C}_{\mathbb{Z}/2}(\mathcal{P}))[[K]^{-1}| H^{\bullet}(K) = 0] \overset\sim\to \mathcal{SDH}_{\mathbb{Z}/2,tw}(\mathcal{A}).
\end{equation}
The algebra on the left hand side is the algebra $\mathcal{DH(A)}$ from Bridgeland's work \cite{B}. It has a set of generators $\left\{E_A, F_A, K_{\alpha}, K_{\alpha}^* | A \in \Iso(\mathcal{A}), \alpha \in K_0(\mathcal{A})\right\},$ cf. \cite{B}.

\begin{proposition}
\begin{itemize}
\item[(i)] The isomorphism 
$$F: \mathcal{SDH}_{\mathbb{Z}/2,tw}(\mathcal{A}) \overset\sim\to \mathcal{DH(A)},$$ 
inverse to $I,$ is defined on the generators in the following way:
\begin{equation}
[\xymatrix{0 \ar@<0.5ex>[r] & A \ar@<0.5ex>[l]}] \mapsto E_A, \quad [\xymatrix{A \ar@<0.5ex>[r] & 0 \ar@<0.5ex>[l]}] \mapsto F_A, \quad [K_\alpha] \mapsto [K_\alpha], \quad [K_\alpha^*] \mapsto [K_\alpha^*].
\end{equation}
\item[(ii)]
The multiplication map 
$$
m: a \otimes b \longmapsto J_{+}^{e}(a) * J_{-}^{e}(b)
$$
defines  an isomorphism of vector spaces 
$$
m: \mathcal{H}_{tw}^{e}(\mathcal{A}) \otimes_{\mathbb{C}} \mathcal{H}_{tw}^{e}(\mathcal{A}) 
     \overset\sim\to \mathcal{SDH}_{\mathbb{Z}/2,tw}(\mathcal{A}).
$$
\end{itemize}
\end{proposition}

Point (i) can be easily checked by hand (it is enough to show that $F$ and $I$ are inverse to each other as maps). Point (ii) follows from {\cite[Lemmas~4.6, 4.7]{B}.

Combining Yanagida's theorem \cite[Theorem~1.26]{Y} with the isomorphism (\ref{dh(a)}), we get another point of view on the semi-derived Hall algebras.

\begin{theorem} \label{drinf}
The algebra $\mathcal{SDH}_{\mathbb{Z}/2,tw}(\mathcal{A})$ is isomorphic to the Drinfeld double of the bialgebra $\mathcal{H}_{tw}^{e}(\mathcal{A}).$
\end{theorem}

This fact combined with Theorem \ref{z2sdhequivtw} yields a new proof of the following theorem of Cramer.

\begin{theorem} [{\cite[Theorem~1]{C}}]
Suppose for two hereditary abelian categories $\mathcal{A'}, \mathcal{A}$ satisfying conditions (C1) and (C2) and having enough projectives, an exact functor $F: \mathcal{A'} \to \mathcal{A}$ induces an equivalence of bounded derived categories
$$F: \mathcal{D}^b(\mathcal{A'}) \overset\sim\to \mathcal{D}^b(\mathcal{A}).$$
Then $F$ induces an algebra isomorphism of the Drinfeld doubles of the bialgebras $\mathcal{H}_{tw}^{e}(\mathcal{A'})$ and $\mathcal{H}_{tw}^{e}(\mathcal{A}).$ 
\end{theorem}

In fact, Cramer proved a more general result: he did not assume that $\mathcal{A'}$ and $\mathcal{A}$ have enough projectives, and the derived equivalence did not have to be induced by an exact functor of abelian categories. By Happel-Reiten-Schofield result \cite[Theorem~1]{HRS} and our Theorem \ref{z2tiltequiv} on tilting invariance, we recover many cases of Cramer's more general statement.

Now consider an acyclic quiver $Q,$ a field $k = \mathbb{F}_q$ with $q$ elements, and the corresponding quantum group $U_{\sqrt{q}}(\mathfrak{g}).$ Thanks to Bridgeland's theorem \cite[Theorem~4.9]{B}, we have the following corollary of Theorem \ref{drinf}.

\begin{proposition} \label{uqgsdhkq}
There is an injective homomorphism of algebras
$$R': U_{\sqrt{q}}(\mathfrak{g}) \hookrightarrow \mathcal{SDH}_{\mathbb{Z}/2,tw,red}(\rep_k Q),$$
defined on generators by
$$R'(E_i) = (q - 1)^{-1} \cdot [\xymatrix{0 \ar@<0.5ex>[r] & S_i \ar@<0.5ex>[l]}], \quad R'(F_i) = (-\sqrt{q})(q - 1)^{-1} \cdot [\xymatrix{S_i \ar@<0.5ex>[r] & 0 \ar@<0.5ex>[l]}],$$
$$R'(K_i) = [K_{S_i}], \quad R'(K_i^{-1}) = [K_{S_i}^*].$$
The map $R'$ is an isomorphism precisely when the graph underlying $Q$ is a simply-laced Dynkin diagram.
\end{proposition}

\subsection{Reflection functors and the braid group action on $U_{\sqrt{q}}(\mathfrak{g})$}

A {\it source} of a quiver $Q$ is a vertex without incoming arrows, a {\it sink} is a vertex without outgoing arrows. Let $i$ be a sink and let $Q'$ be the quiver obtained from $Q$ by reversing all arrows starting in $i.$ Let us denote by $S_j$ and $S_j'$ the simple objects concentrated at a vertex $j$ in $\rep_k (Q)$ and $\rep_k (Q'),$ respectively. Let $\tau^{-} (S_i)$ be the cokernel of the natural monomorphism
$$S_i =  P_i \hookrightarrow \bigoplus\limits_{j \to i} P_j.$$
Then 
$$T := \bigoplus\limits_{j \neq i} P_j \oplus \tau^{-} (S_i)$$
is a tilting object in $\rep_k (Q),$ and the indecomposable objects of the category $\Fac T$ of all quotients of finite direct sums of copies of $T$ are all the indecomposables of $\rep_k (Q),$ except for $S_i.$ In \cite{SV}, this category is denoted by $\rep_{adm} (Q).$ Define $\mathcal{T}' = \rep_{adm} (Q')$ in a similar way, i.e. as the full exact subcategory containing all the indecomposables of $\rep_k (Q'),$ except for $S_i'.$ There is an equivalence of categories 
$$s_i: \rep_{adm} (Q) \overset\sim\to \rep_{adm} (Q'),$$
cf. \cite{SV}. The functors $s_i$ are called {\it reflection functors}, since $s_i$ induces the action of a simple reflection in the Weyl group of $Q$ on the lattice $\mathbb{Z}^n$ which can be identified with both Grothendieck groups $K_0(\rep_k (Q)$ and $K_0(\rep_k (Q')).$ For different $i,$ these functors  generate the action of the whole Weyl group on the lattice. 
We have a diagram of functors
\begin{equation} \label{bgpfact}
\rep_k (Q) \hookleftarrow \Fac T \overset\sim\to \mathcal{T}' \hookrightarrow \rep_k (Q').
\end{equation}
All of them induce derived equivalences, and we can finally define an equivalence of the bounded derived categories
$$s_i^d: \mathcal{D}^b(\rep_k (Q)) \overset\sim\to \mathcal{D}^b(\rep_k (Q'))$$
to be their composition (from left to right). These functors are the derived versions of the {\it Bernstein-Gelfand-Ponomarev reflection functors} \cite{BGP}. By computing the effect of the functors (\ref{bgpfact}) in the semi-derived Hall algebra we obtain the following theorem.

\begin{theorem} \label{bgp}
There is a unique algebra isomorphism 
$$t_i: \mathcal{SDH}_{\mathbb{Z}/2,tw,red}(\rep_k (Q)) \overset\sim\to \mathcal{SDH}_{\mathbb{Z}/2,tw,red}(\rep_k (Q')),$$
satisfying
$$* \circ t_i = t_i \circ *;$$
\begin{equation} \label{bgpsi}
t_i([\xymatrix{S_i \ar@<0.5ex>[r] & 0 \ar@<0.5ex>[l]}]) = q^{-1/2} \cdot [\xymatrix{0 \ar@<0.5ex>[r] & S_i' \ar@<0.5ex>[l]}] * K_{S_i'}^*;
\end{equation}
$$t_i([\xymatrix{A \ar@<0.5ex>[r] & 0 \ar@<0.5ex>[l]}]) = [\xymatrix{s_i(A) \ar@<0.5ex>[r] & 0 \ar@<0.5ex>[l]}], \quad \forall A \in \Fac T;$$
$$t_i(K_\alpha) = K_{s_i \alpha}. $$
This isomorphism is induced by the derived equivalence $s_i^d.$
\end{theorem}

\begin{proof}
Only equation (\ref{bgpsi}) is not straightforward. To prove this, note that we have a triangle 
$$S_i  \hookrightarrow \bigoplus\limits_{j \to i} P_j \to \tau^{-} (S_i) \to \Sigma S_i$$
in the derived category $\mathcal{D}^b(\rep_k (Q)) \overset\sim\to \mathcal{D}^b(\rep_k (Q')).$ The first three objects here belong to the category $\rep_k (Q),$ while the images of the last three under $s_i^d$ belong to $\rep_k (Q').$ We thus have conflations 
$$\xymatrix{{(S_i} \ar@<0.5ex>[r] & {0)\quad} \ar@<0.5ex>[l]  \ar@{>->}[r]  & \quad(\bigoplus\limits_{j \to i} P_j \ar@<0.5ex>[r] & \tau^{-} (S_i))\quad \ar@<0.5ex>[l]^{0}  \ar@{->>}[r]  & \quad{K_{\tau^{-} (S_i)}}},$$
$$\xymatrix{{K_{s_i(\bigoplus\limits_{j \to i} P_j)}\quad} \ar@{>->}[r]  & \quad(s_i(\bigoplus\limits_{j \to i} P_j) \ar@<0.5ex>[r] & s_i(\tau^{-} (S_i)))\quad \ar@<0.5ex>[l]^{0} \ar@{->>}[r] & \quad(0 \ar@<0.5ex>[r] & S_i') \ar@<0.5ex>[l]}$$
in $\rep_k (Q)$ and $\rep_k (Q'),$ respectively. Formula (\ref{bgpsi}) follows by simple calculations.
\end{proof}

Sevenhant-Van den Bergh \cite{SV} consider the {\it reduced Drinfeld double} $U(Q)$ of the algebra $\mathcal{H}_{tw}^{e} (\rep_k (Q)),$ where reduction is given by imposing (\ref{reduction}). By our results, the algebra $U(Q)$ is isomorphic to $\mathcal{SDH}_{\mathbb{Z}/2,tw,red}(\rep_k (Q)).$ Therefore, Theorem~9.1 of \cite{SV} is a corollary of the above theorem.

As shown in \cite{SV} and \cite{XY}, by combining the isomorphisms $t_i$ with an appropriate Fourier transform, one gets an induced action of the braid group of the quiver $Q$ on the quantum group $U_{\sqrt{q}}(\mathfrak{g}),$ identified with a subalgebra of $\mathcal{SDH}_{\mathbb{Z}/2,tw,red}(\rep_k (Q))$ by Proposition \ref{uqgsdhkq}. This Fourier transform sends $S_j$ to $S_j'$  \cite{XY}, and its action on the $\mathbb{Z}/2-$graded semi-derived Hall algebras is naturally induced. As shown on \cite{SV} and \cite{XY}, this braid group action coincides with the one defined by Lusztig by different methods \cite{L}. 

\end{document}